%% file: paper_plain.tex
\documentclass{article}
\usepackage{geometry} 
\usepackage{lipsum}
\usepackage{amsfonts}
\usepackage{amssymb}
\usepackage{mathrsfs}
\usepackage{graphicx}
\usepackage{epstopdf}
\usepackage[utf8]{inputenc}
\usepackage{amsmath}
\usepackage{fancyhdr}
\usepackage{amsthm}
\usepackage[usenames]{color}
\usepackage[tableposition=top,figureposition=top]{caption}
\usepackage{subcaption}
\usepackage{wrapfig}
\usepackage{float}
\usepackage{txfonts}
\usepackage{enumerate}
\usepackage{etoolbox}
\usepackage{verbatim} 
\usepackage{pifont}
\usepackage{xifthen}
\usepackage{algorithmic}
\usepackage{amsopn}

\input{macros.tex}

\newcommand{\OI}{\operatorname{\mathsf{I}}}
\newcommand{\norm}[2]{\left\lVert #1\right\rVert_{#2} }

\ifpdf
  \DeclareGraphicsExtensions{.eps,.pdf,.png,.jpg}
\else
  \DeclareGraphicsExtensions{.eps}
\fi

\newtheorem{theorem}{Theorem}[section]
\newtheorem{lemma}[theorem]{Lemma}
\newtheorem{proposition}[theorem]{Proposition}
\newtheorem{remark}[theorem]{Remark}
\newtheorem{definition}[theorem]{Definition}

\title{A new approach to space-time boundary integral equations for the
  wave equation}
\date{}
\author{
Olaf~Steinbach\thanks{Institute of Applied Mathematics, TU Graz, 
  Austria}
\and Carolina~Urz\'ua-Torres\thanks{Delft Institute of Applied Mathematics,
  TU Delft, The Netherlands. \newline This publication arises from research funded 
  by the John Fell Oxford University Press Research Fund.}
  }

\numberwithin{equation}{section}

\begin{document}

\maketitle

\begin{abstract}
 We present a new approach for boundary integral equations for the wave equation 
 with zero initial conditions. Unlike previous attempts, our mathematical 
 formulation allows us to prove that the associated boundary integral operators 
 are continuous and satisfy inf-sup conditions in trace spaces of the same 
 regularity, which are closely related to standard energy spaces with the 
 expected regularity in  space and time. This feature is crucial from a 
 numerical perspective, as it  provides the foundations to derive sharper 
 error estimates and paves the way  to devise efficient adaptive space-time 
 boundary element methods, which will be tackled in future work. 
 On the other hand, the proposed approach is compatible with current time 
 dependent boundary element method's implementations and we predict that it 
 explains many of the behaviours observed in practice but that were not 
 understood with the existing theory.
\end{abstract}

\section{Introduction}
\label{sec:introduction}

Different strategies have been used to derive variational methods for time 
domain boundary integral equations for the wave equation. 
The more established and successful ones include weak formulations derived 
via the Laplace transform, and also space-time energetic variational 
formulations, often referred as \emph{energetic BEM} in the literature. 
These approaches started with the groundbreaking works of Bamberger and 
Ha Duong \cite{BHD86}, and Aimi et al. \cite{ADG08}, respectively. 
In spite of their extensive use \cite{ADG11, ADG09, BGN16, GNS17, GOS19,
  GOS20, GOS18, HAD03,  HQS17, JoR17,  PoS21, Say13, Say16}
at the time of writing this article, 
the numerical analysis corresponding to these formulations was still incomplete 
and presents difficulties that are hard to overcome, if possible at all. 

One of these difficulties is the fact that current approaches provide continuity 
and coercivity estimates which are not in the same space-time (Sobolev) norms. 
Indeed, there is a so-called \emph{norm gap} arising from a loss of regularity 
in time of the related boundary integral operators. Yet, recent work by Joly 
and Rodr\'iguez shows that these norm gaps are not present in 1D \cite{JoR17}. 
Moreover, to the best of the authors' knowledge, there is no proof nor numerical 
evidence that such loss of time regularity should hold for higher dimensions 
either. These two observations encouraged us to believe that one may be able to 
prove sharper results using different mathematical tools.
Another disadvantage of current strategies is that they do not provide the 
foundations for space-time boundary element methods, which are basically 
boundary element discretizations where the time variable is treated simply as 
another space variable, in contrast to techniques such as time-stepping methods 
and convolution quadrature methods \cite{Say16}.

Space-time discretization methods offer an increasingly popular alternative, 
since they allow the treatment of moving boundaries, adaptivity in space and 
time simultaneously, and space-time parallelization
\cite{GaN16,ScS09,Ste15,StY18}. 
However, in order to exploit these advantages, one needs to have a complete numerical 
analysis of the corresponding Galerkin methods.

We construct a new approach to boundary integral equations for the wave 
equation by working directly in the time domain. Furthermore, we develop a 
mathematical framework that not only overcomes the aforementioned 
difficulties, but also paves the way to stable space-time FEM/BEM coupling.
We present these new results following the standard pieces and arguments 
from classical boundary integral equations. We hope this highlights some 
mathematical intuitions behind the obtained results and makes the article 
easier to read for those familiarized with the boundary integral equation 
literature.
In addition to a new boundary integral equation formulation, we provide novel 
existence and uniqueness results for the Dirichlet and Neumann wave equation 
initial boundary value problems, when initial conditions are zero.

The structure of this article is as follows. 
Section~\ref{sec:preliminaries} introduces notation and summarizes results from 
the literature that will be needed later in the paper. We begin by using some 
key ideas of recent work on the wave equation in $H^1(Q)$ \cite{StZ20,Zan19}. 
Then, in Section~\ref{sec:Traces}, we introduce trace spaces, trace operators 
and their corresponding properties for three different families of spaces. With 
this we aim, on the one hand, to emphasize the link between the existing 
space-time (volumetric) variational formulations and our 
new results. On the other hand, we 
prove that the related trace spaces are indeed connected, which provides a new 
and deeper understanding of the different existing boundary integral formulations 
and their relation. Section~\ref{sec:BVPs} presents some required results on 
initial boundary value problems for the wave equation, while all the remaining 
building blocks of the new boundary integral equation formulation are presented 
in Section~\ref{sec:BIE}. 
This final section concludes with the existence and uniqueness 
results for solutions of related boundary integral equations.

\section{Preliminaries}
\label{sec:preliminaries}

\subsection{Model problem}
\label{ssec:modelprob}
Let $\Omega \subset \R^n$, $n=1,2,3$, with boundary $\Gamma :=\partial \Omega$.
We assume $\Omega$ to be an interval for $n=1$, or a bounded Lipschitz domain
for $n=2,3$. Let $0<T<\infty$. For a finite time interval $(0,T)$, we define
the \emph{space-time cylinder}
$Q := \Omega \times (0,T) \, \subset \, \R^{n+1}$,
and its lateral boundary $\Sigma := \Gamma \times [0,T]$. We also introduce 
the initial boundary $\Sigma_0 := \Omega \times \{0\}$, and the final boundary 
$\Sigma_T := \Omega \times \{T\}$.
We denote the D'Alembert operator by $\Box := \partial_{tt} - \Delta_x$,
and write the \emph{interior Dirichlet initial boundary value problem
for the wave equation} as
\begin{equation}\label{eq:IBVP}
\begin{array}{rclcl}
\Box u(x,t) & = & f(x,t) & & \text{for} \; (x,t) \in Q, \\
u(x,t) & = & g(x,t) & & \text{for} \; (x,t) \in \Sigma,\\
u(x,0) = \partial_t u(x,t)_{|t=0} & = & 0 & & \text{for} \; x \in \Omega.
\end{array}
\end{equation}

\subsection{Notation and mathematical framework}
\label{ssec:notation}
Let ${\mathcal{O}} \subseteq \R^m$, $m\in \N$.
We stick to the usual notation for the space $C^\infty({\mathcal{O}})$
of functions which are bounded and infinitely often continuously
differentiable; the subspace $C^\infty_0({\mathcal{O}})$ of
compactly supported smooth functions; the spaces $L^p({\mathcal{O}})$
of Lebesgue integrable functions; and the Sobolev spaces $H^s({\mathcal{O}})$. 
Moreover, inner products of Hilbert spaces $X$ are denoted by standard
brackets $(\cdot,\cdot)_X$, while angular brackets
$\dual{\cdot}{\cdot}_{\mathcal{O}}$ are used for the duality pairing
induced by the extension of the inner product
$(\cdot, \cdot)_{L^2({\mathcal{O}})}$. For a Hilbert space $X$ we denote
by $X'$ its dual with the norm
\[
  \| f \|_{X'} = \sup\limits_{0 \neq v \in X}
  \frac{|\langle f , v \rangle_{\mathcal{O}}|}{\| v \|_X} \quad
  \mbox{for} \; f \in X' .
\]
In particular, we will use
\[
  H^1({\mathcal{O}}) :=
  \overline{C^\infty({\mathcal{O}})}^{\| \cdot \|_{H^1({\mathcal{O}})}}, \quad
  H^1_0({\mathcal{O}}) :=
  \overline{C^\infty_0({\mathcal{O}})}^{\| \cdot \|_{H^1({\mathcal{O}})}},
\]
where
\[
  \| \phi \|_{H^1({\mathcal{O}})} := \left(
    \|\phi \|^2_{L^2({\mathcal{O}})} + \sum\limits_{i=1}^m
    \| \partial_{x_i} \phi \|^2_{L^2({\mathcal{O}})}
  \right)^{1/2} \, .
\]
In the specific case ${\mathcal{O}} = Q = \Omega \times (0,T) \subset
{\mathbb{R}}^{n+1}$ we identify $H^1(Q)$ with the Sobolev space
\[
H^{1,1}(Q) := L^2(0,T;H^1(\Omega)) \cap H^1(0,T;L^2(\Omega))
\]
using Bochner spaces, see, e.g.,
\cite[Sect.~1.3, Chapt.~1]{LIM72i} and \cite[Sect.~2, Chapt.~4]{LIM72ii}.
Furthermore, let
\begin{eqnarray*}
  H^1_{0,}(0,T;L^2(\Omega))
  & := & \Big \{
  v \in L^2(Q) : \partial_t v \in L^2(Q), \; v(x,0)=0 \quad \mbox{for} \;
  x \in \Omega \Big \}, \\
  H^1_{,0}(0,T;L^2(\Omega))
  & := & \Big \{
  v \in L^2(Q) : \partial_t v \in L^2(Q), \; v(x,T)=0 \quad \mbox{for} \;
  x \in \Omega \Big \} .
\end{eqnarray*}
With this we introduce
\begin{eqnarray*}
  H^{1,1}_{;0,}(Q)
  & := & L^2(0,T;H^1(\Omega)) \cap H^1_{0,}(0,T;L^2(\Omega)), \\
  H^{1,1}_{;,0}(Q)
  & := & L^2(0,T;H^1(\Omega)) \cap H^1_{,0}(0,T;L^2(\Omega)),
\end{eqnarray*}
with norms
\begin{eqnarray*}
  \| u \|_{H^{1,1}_{;0,}(Q)} := \sqrt{\| \partial_t u \|_{L^2(Q)}^2
         + \| \nabla_x u \|^2_{L^2(Q)}} \, , \\
  \| v \|_{H^{1,1}_{;,0}(Q)} := \sqrt{\| \partial_t v \|_{L^2(Q)}^2
         + \| \nabla_x v \|^2_{L^2(Q)}} \, .
\end{eqnarray*}
Note that the space $H^{1}_{;0,}(Q)$ corresponds to
$\vphantom{,}_{0}H^1(Q)$ as used in \cite{HAD03, LIM72i, LIM72ii}.
In the case of zero Dirichlet boundary data along the lateral boundary $\Sigma$
we define the subspaces
\begin{eqnarray*}
  H^{1,1}_{0;0,}(Q)
  & := & L^2(0,T;H^1_0(\Omega)) \cap H^1_{0,}(0,T;L^2(\Omega)), \\
  H^{1,1}_{0;,0}(Q)
  & := & L^2(0,T;H^1_0(\Omega)) \cap H^1_{,0}(0,T;L^2(\Omega)).
\end{eqnarray*}
We remark that $H^{1,1}_{;0,}(Q)$ and $H^{1,1}_{0;0,}(Q)$ prescribe zero initial
values at $t=0$, while $H^{1,1}_{;,0}(Q)$ and $H^{1,1}_{0;,0}(Q)$ have zero
final values at $t=T$. 

In this paper we will consider, as in
\cite{StZ21}, a generalized variational formulation to describe
solutions of the wave equation \eqref{eq:IBVP} also for
$f \in [H^{1,1}_{0;,0}(Q)]'$, instead of $f \in L^2(Q)$, as usually
considered, e.g., \cite{Lad85}. Therefore we introduce the \emph{extended}
space-time cylinder $Q_- := \Omega\times(-T,T)$. For $u \in L^2(Q)$, we define
$\widetilde{u} \in L^2(Q_-)$ as zero extension,
\[
  \widetilde{u}(x,t) := \left \{
    \begin{array}{ccl}
      u(x,t) & & \mbox{for} \; (x,t) \in Q, \\
      0 & & \mbox{for} \; (x,t) \in Q_- \backslash Q .
    \end{array} \right.
\]
The application of the wave operator $\Box$ to $\widetilde{u} \in L^2(Q_-)$
is defined as a distribution on $Q_-$, i.e., for all test functions
$\varphi \in C^\infty_0(Q_-)$, we define 
\[
  \langle \Box \widetilde{u} , \varphi \rangle_{Q_-} :=
  \int_{-T}^T \int_\Omega \widetilde{u}(x,t) \,
  \Box \varphi(x,t) \, dx \, dt =
  \int_0^T \int_\Omega u(x,t) \,
  \Box \varphi(x,t) \, dx \, dt \, .
\]
This motivates to consider the Sobolev space $H^1_0(Q_-)$ with the
norm
\[
  \| \phi \|_{H^1_0(Q_-)} =
  \sqrt{\| \partial_t \phi \|^2_{L^2(Q_-)} +
  \| \nabla_x \phi \|_{L^2(Q_-)}^2} \quad \mbox{for} \; \phi \in H^1_0(Q_-),
\]
the dual space $[H^1_0(Q_-)]'$, and the duality pairing
$\langle \cdot,\cdot \rangle_{Q_-}$ as extension of the inner product
in $L^2(Q_-)$. We also introduce the restriction operator
${\mathcal{R}} : H^1_0(Q_-) \to H^{1,1}_{0;,0}(Q)$, i.e.,
${\mathcal{R}} \phi := \phi_{|Q}$, and its adjoint
${\mathcal{R}}' : [H^{1,1}_{0;,0}(Q)]' \to [H^1_0(Q_-)]'$.
Moreover, let ${\mathcal{E}} : H^{1,1}_{0;,0}(Q) \to H^1_0(Q_-)$
be any continuous and injective extension operator with norm
\[
  \| {\mathcal{E}} \|_{H^{1,1}_{0;,0}(Q),H^1_0(Q_-)} :=
  \sup\limits_{0 \neq v \in H^{1,1}_{0;,0}(Q)}
  \frac{\| {\mathcal{E}}v \|_{H^1_0(Q_-)}}{\| v \|_{H^{1,1}_{0;,0}(Q)}} \, ,
\]
and its adjoint
${\mathcal{E}}' : [H^1_0(Q_-)]' \to [H^{1,1}_{0;,0}(Q)]'$, satisfying
${\mathcal{R}}{\mathcal{E}} \phi = \phi$ for all $\phi \in H^{1,1}_{0;,0}(Q)$.

As in \cite{StZ21} we introduce the Banach space
\[
  {\mathcal{H}}(Q) := \Big \{
  u = \widetilde{u}_{|Q} : \widetilde{u} \in L^2(Q_-), \;
  \widetilde{u}_{|\Omega \times (-T,0)} = 0, \;
  \Box \widetilde{u} \in [H^1_0(Q_-)]' \Big \},
\]
with the norm
$
  \| u \|_{{\mathcal{H}}(Q)} :=
  \sqrt{ \| u \|_{L^2(Q)}^2 + \| \Box \widetilde{u} \|^2_{[H^1_0(Q_-)]'}} \, ,
$
where 
\[\| \Box \widetilde{u} \|_{[H^1_0(Q_-)]'} =
  \sup\limits_{0 \neq v \in H^{1,1}_{0;,0}(Q)}
  \frac{|\langle \Box \widetilde{u} , {\mathcal{E}}v\rangle_{Q_-}|}
  {\| v \|_{H^{1,1}_{0;,0}(Q)}} \,.
  \]
By completion, we finally define the Hilbert spaces
\[
  {\mathcal{H}}_{0;0,}(Q) :=
  \overline{H^{1,1}_{0;0,}(Q)}^{\| \cdot \|_{{\mathcal{H}}(Q)}} \subset
  {\mathcal{H}}_{;0,}(Q) :=
  \overline{H^{1,1}_{;0,}(Q)}^{\| \cdot \|_{{\mathcal{H}}(Q)}} \subset
  {\mathcal{H}}(Q),
\]
e.g.,
\[
  {\mathcal{H}}_{;0,}(Q) = \Big \{
  u \in {\mathcal{H}}(Q) : \exists (u_n)_{n \in {\mathbb{N}}} \subset
  H^{1,1}_{;0,}(Q) \; \mbox{with} \; \lim\limits_{n \to \infty}
  \| u - u_n \|_{{\mathcal{H}}(Q)} = 0 \Big \} .
\]
Note that $H^{1,1}_{0;0,}(Q) \subset {\mathcal{H}}_{0;0,}(Q)$ and
$H^{1,1}_{;0,}(Q) \subset {\mathcal{H}}_{;0,}(Q)$, see
\cite[Lemma 3.5]{StZ21} for the first inclusion.

\subsection{Transformation operator $\calH_T$}
\label{ssec:HilbertTh}

For $u \in L^2(0,T)$ we consider the Fourier series
\begin{align*}
 u(t) = \sum_{k=0}^\infty u_k \sin\left(\left(\frac{\pi}
 {2} + k\pi\right)\frac{t}{T}\right), \quad 
  u_k = \frac{2}{T}\int_0^T
  u(t)\sin\left(\left(\frac{\pi}{2} + k\pi\right)\frac{t}{T} \right) \, dt .
\end{align*}
As in \cite{StZ20} we introduce the transformation operator $\calH_T$ as
\begin{align*}
\calH_T u(t) := \sum_{k=0}^\infty u_k \cos\left(\left(
\frac{\pi}{2} + k\pi\right)\frac{t}{T}\right),
\end{align*}
and it's inverse, i.e., for $v \in L^2(0,T)$,
\begin{align*}
  \calH_T^{-1} v(t) := \sum_{k=0}^\infty
  \overline{v}_k \sin\left(
  \left(\frac{\pi}{2} + k\pi\right)\frac{t}{T}\right), \quad
\overline{v}_k = \frac{2}{T}\int_0^T
  v(t)\cos\left(\left(\frac{\pi}{2} + k\pi\right)\frac{t}{T} \right) \, dt .
\end{align*}
By construction we have
$\calH_T : \, H^1_{0,}(0,T) \to H^1_{,0}(0,T)$, and
$\calH_T^{-1} : \, H^1_{,0}(0,T) \to H^1_{0,}(0,T)$.
In the following, we summarize some additional properties fulfilled by the 
operators $\calH_T$ and $\calH_T^{-1}$, see \cite{StZ20,Zan19}.

\pagebreak

\begin{proposition}\label{proposition Hilbert}
\hspace*{1cm}  
\begin{enumerate}
\item For any $u,v \in L^2(0,T)$
\begin{align*}
  \langle \calH_T u , v \rangle_{L^2(0,T)} =
  \langle u , \calH_T^{-1} v \rangle_{L^2(0,T)}.
\end{align*}
\item For all $u \in H^1_{0,}(0,T)$ 
\begin{align*}
\partial_t \calH_T u = - \calH_T^{-1} \partial_t u.
\end{align*}
\item $\calH_T$ and $\calH_T^{-1}$ are norm preserving, i.e.,
\begin{align*}
  \| \calH_T w \|_{L^2(0,T)} = \| w \|_{L^2(0,T)},  \quad
  \| \calH_T^{-1} w \|_{L^2(0,T)} = \| w \|_{L^2(0,T)} \quad
  \forall w \in L^2(0,T).
\end{align*}
\item For all $w \in L^2(Q)$
\begin{align*}
 \langle w , \calH_T w \rangle_{L^2(0,T)} \geq 0.
\end{align*}
\end{enumerate}
\end{proposition}

\noindent
We conclude this subsection by extending the modified Hilbert transformation
${\mathcal{H}}_T$ to our functional spaces. For $u \in L^2(Q)$ we first
have the decomposition
\begin{align*}
 u(x,t) = \sum_{k=0}^\infty\sum_{i=0}^\infty u_{i,k} \sin\left(\left(\frac{\pi}
 {2} + k\pi\right)\frac{t}{T}\right) \, \varphi_i(x), \\
  u_{i,k} = \frac{2}{T}\int_Q u(x,t) \,
  \sin\left(\left(\frac{\pi}{2} + k\pi\right)\frac{t}{T} \right) \, dt
  \, \varphi_i(x) \, dx ,
\end{align*}
where $\varphi_i$ are the Neumann eigenfunctions of the Laplacian, i.e.,
\begin{align*}
-\Delta \varphi_i = \lambda_i \varphi_i \text{ in }\Omega, \quad
\partial_{n_x}\varphi_i = 0 \quad\text{ on }\Gamma, \quad
\norm{\varphi_i}{L^2(\Omega)} = 1, \quad 0 = \lambda_0 < \lambda_i \quad
\forall i \in \N .
\end{align*}
They are an orthonormal basis in $L^2(\Omega)$ and an orthogonal basis in
$H^1(\Omega)$, e.g., \cite[Chapt.~2]{Lad85}. With this we define
\begin{align*}
  \calH_T u(x,t) :=
  \sum_{k=0}^\infty\sum_{i=0}^\infty u_{i,k} \cos\left(\left(
  \frac{\pi}{2} + k\pi\right)\frac{t}{T}\right)\varphi_i(x), \quad
  (x,t) \in Q,
\end{align*}
with ${\calH}_T: H^{1,1}_{;0,}(Q) \to H^{1,1}_{;,0}(Q)$.
Analogously,
${\calH}_T^{-1} : H^{1,1}_{;,0}(Q) \to H^{1,1}_{;0,}(Q)$. 

\begin{remark}
The time-reversal map $\kappa_T$, defined as \cite[Eq.~(2.36)]{COS90} 
\begin{equation}\label{eq:kt}
 \kappa_T \,w (x,t) := w(x,T-t) \quad \mbox{for} \; (x,t) \in Q ,
 w \in H^1(Q),
\end{equation}
is often used instead of the transformation operator 
${\calH}_T: H^{1,1}_{;0,}(Q) \to H^{1,1}_{;,0}(Q)$.
\end{remark}

\subsection{Fundamental solution and retarded potentials}
Let us briefly present the \emph{boundary layer potentials} for the wave 
equation, often called \emph{retarded potentials}. We refer the reader to 
\cite{COS94} and \cite{HAD03} for further details. 
First, we introduce the fundamental solution of the wave equation,
\begin{align}\label{eq:fundsol}
 G(x,t) =\begin{cases}
 \dfrac{1}{2} \, \mathsf{H}(t-\abs{x}), & \quad n=1,\\[0.2in]
 \dfrac{1}{2\pi} \, \dfrac{\mathsf{H}(t-\abs{x})}{\sqrt{t^2-\abs{x}^2}},
 & \quad n=2,\\[0.3in]
  \dfrac{1}{4\pi} \, \dfrac{\delta(t - \abs{x})}{\abs{x}},& \quad n=3,
         \end{cases}
\end{align}
with $\delta$ the Dirac distribution, and $\mathsf{H}$
the Heaviside step function.
Let $\mathscr{S}$ be the \emph{single layer potential} and $\mathscr{D}$ the 
\emph{double layer potential}, i.e., for $(x,t) \in Q$ and
regular enough densities $w$ and $z$, respectively,
\begin{align}\label{eq:SL}
 (\mathscr{S}w)(x,t) &:= \int_0^t \int_\Gamma G(x-y,t-\tau)
 \, w(y,\tau) \, ds_y \, d\tau,\\ \label{eq:DL}
  (\mathscr{D}z)(x,t) &:= \int_0^t \int_\Gamma
 \partial_{n_y}  G(x-y,t-\tau) \, z(y,\tau) \, ds_y \, d\tau.
\end{align}
Concretely, for $n=3$, these are
\begin{align}
 (\mathscr{S}w)(x,t) &:= \frac{1}{4\pi} \int_\Gamma \frac{w(y,
 t-\abs{x-y})}{\abs{x-y}} ds_y , \\
 (\mathscr{D}z)(x,t) &:= \frac{1}{4\pi} \int_\Gamma \left[ \partial_{n_y} 
 \frac{z(y,t-\abs{x-y})}{\abs{x-y}}  - \frac{\partial_{n_y}\abs{x-y}}{\abs{
 x-y}}\partial_t z(y,t-\abs{x-y}) \right] ds_y.
\end{align}
The fact that the time argument is the retarded time $\tau = t-\abs{x-y}$ 
motivates that $\mathscr{S}$ and $\mathscr{D}$ are usually called retarded 
potentials.

\section{Green's Formula, Trace Spaces and Trace Operators}
\label{sec:Traces}
We introduce the lateral interior trace operator
$\gamma_\Sigma^i : u \mapsto u_{|\Sigma}$
as continuous extension of the trace map defined in the pointwise sense for
smooth functions. 
As in \cite[Lemma 4.1]{McL00} we can write a space-time Green's
formula for $ \varphi \in C^2(Q)$ and $\psi \in C^1(Q)$ as
\[
  \Phi(\varphi,\psi) =
  \int_0^T \int_\Omega \Box \varphi \, \psi \, dx \, dt
  + \int_0^T \int_\Gamma \partial_{n_x} \varphi \,
  \gamma_\Sigma^i \psi \, ds_x \, dt -
  \int_\Omega \Big[ \partial_t \varphi \, \psi \Big]_{t=0}^T \, dx,
\]
where
\begin{equation}\label{Def Phi}
  \Phi(\varphi,\psi) :=
  - \int_0^T \int_\Omega \partial_t \varphi \, \partial_t \psi \, dx \, dt +
  \int_0^T \int_\Omega \nabla_x \varphi \cdot \nabla_x \psi \, dx \, dt \, .
\end{equation}
In particular, for $ \varphi \in C^2(Q)$ with $\partial_t \varphi(x,t)_{|t=0}=0$
for $x \in \Omega$ and for $ \psi \in C^1(Q)$ with $\psi(x,T)=0$ for
$x \in \Omega$, this gives Green's first formula
\begin{equation}\label{Green 1 smooth}
\Phi(\varphi,\psi) =
  \int_0^T \int_\Omega \Box \varphi \, \psi \, dx \, dt
  + \int_0^T \int_\Gamma \partial_{n_x} \varphi \,
  \gamma_\Sigma^i \psi \, ds_x \, dt .
\end{equation}

\subsection{Traces on $H^{1,1}_{;0,}(Q)$, $H^{1,1}_{;,0}(Q)$, and
${\mathcal{H}}_{;0,}(Q)$}
\label{ssec:TracesQ}

Following \cite[Theorem~2.1, Chapt.~4 and p.~19]{LIM72ii} we get that the
interior trace map $\gamma_{\Sigma}^i$ is continuous and surjective
from $H^1(Q)$ to $H^{1/2}(\Sigma)$. In addition, let
${\mathcal{E}}_\Sigma : H^{1/2}(\Sigma) \to H^1(Q)$ be a continuous
right inverse.

Let us introduce the spaces
\begin{align*}
  H^{1/2}_{0,}(\Sigma) := L^2(0,T;H^{1/2}(\Gamma))\cap
  H^{1/2}_{0,}(0,T;L^2(\Gamma)),\\
  H^{1/2}_{,0}(\Sigma) := L^2(0,T;H^{1/2}(\Gamma))\cap
  H^{1/2}_{,0}(0,T;L^2(\Gamma)),
\end{align*}
with $H^{1/2}_{0,}(0,T;L^2(\Gamma))$ and
$H^{1/2}_{,0}(0,T;L^2(\Gamma))$ defined 
by interpolation as
\begin{align*}
  H^{1/2}_{0,}(0,T;L^2(\Gamma)) :=
  [H^1_{0,}(0,T;L^2(\Gamma), L^2(0,T;L^2(\Gamma)]_{1/2},\\
  H^{1/2}_{,0}(0,T;L^2(\Gamma)) :=
  [H^1_{,0}(0,T;L^2(\Gamma), L^2(0,T;L^2(\Gamma)]_{1/2}.
\end{align*}
Then, we have the following result, which is stated in \cite{HAD03} 
without a proof. Here we provide one for completeness.

\begin{lemma}\label{lemma:traceH10I}
  The interior trace map $\gamma_{\Sigma}^i$ is continuous and surjective from
  $H^{1,1}_{;0,}(Q)$ to $H^{1/2}_{0,}(\Sigma)$.
\end{lemma}

\begin{proof}
We adapt the proof of \cite[Theorem~2.1, Chapt.~4]{LIM72ii} 
to $H^{1,1}_{;0,}(Q)$ (instead of $H^1(Q)$). Recall that 
\begin{equation*}
  u \in H^{1,1}_{;0,}(Q) =
  L^2(0,T; H^1(\Omega))\cap H^1_{0,}(0,T;L^2(\Omega)). 
\end{equation*}
Without loss of generality, we can take
$\Omega = \{ x \in \R^n \, : \, x_n>0 \}$ and
$\Gamma = \{ x \in \R^n \, : \, x_n=0 \}$. Then,
by using the notation
$ x = \lbrace x^\prime, x_n \rbrace$, with
$x^\prime = \lbrace x_1, \ldots , x_{n-1} \rbrace$, we can write:
\begin{align*}
  u \in H^{1}_{;0,}(Q) \; \Leftrightarrow \;
  & u \in L^2(\R_{+,x_n}; L^2(0,T;H^1(\R^{n-1}_{x^\prime})) 
  \cap L^2(0,T;L^2(\R^{n-1}_{x^\prime})), \\
  & u \in L^2(\R_{+,x_n}; L^2(0,T;L^2(\R^{n-1}_{x^\prime})) 
  \cap H^1_{0,}(0,T;L^2(\R^{n-1}_{x^\prime})).
\end{align*}
Then, we can apply Theorem~4.2 from \cite[Chapt.~1]{LIM72i} with
\begin{align*}
  X = L^2(0,T; H^1(\R^{n-1}_{x^\prime})) \cap
  H^1_{0,}(0,T;L^2(\R^{n-1}_{x^\prime})),\quad
  Y = L^2(0,T; L^2(\R^{n-1}_{x^\prime})), 
\end{align*}
to get that $u(x^\prime,0,t) \in [X,Y]_{1/2}$.
Now, let us point out that Theorem~13.1 in \cite[Chapt.~1]{LIM72i} gives
\begin{align*}
[X,Y]_{1/2} &= [L^2(0,T; H^1(\R^{n-1}_{x^\prime})) \cap H^1_{0,}(0,T;L^2(\R^{n-
1}_{x^\prime})), Y]_{1/2} \\ &= 
[L^2(0,T; H^1(\R^{n-1}_{x^\prime})), Y]_{1/2}\cap [H^1_{0,}(0,T;L^2(\R^{n-1}_{
x^\prime})) , Y ]_{1/2}.
\end{align*}
Consequently, by interpolation we get
\begin{align*}
[X,Y]_{1/2} = L^2(0,T; H^{1/2}(\R^{n-1}_{x^\prime})) \cap H^{1/2}_{0,}(0,T;L^2(
\R^{n-1}_{x^\prime})),
\end{align*}
which corresponds to $H^{1/2}_{0,}(\R^{n-1}_{x^\prime}\times[0,T])$. 
Hence, we conclude that $\gamma_{\Sigma}^iu \in H^{1/2}_{0,}(\Sigma)$.
Surjectivity also follows from Theorem~4.2 in \cite[Chapt.~1]{LIM72i}.
\end{proof}

\noindent
By similar arguments, one can also prove:

\begin{lemma}\label{lemma:traceH10F}
  The interior trace map $\gamma_{\Sigma}^i$ is continuous and surjective from
  $H^{1,1}_{;,0}(Q)$ to $H^{1/2}_{,0}(\Sigma)$.
\end{lemma}

\noindent
Finally, we define the lateral trace space
\[
  {\mathcal{H}}_{0,}(\Sigma) := \Big \{ v = \gamma_\Sigma^i V \quad
  \mbox{for all} \; V \in {\mathcal{H}}_{;0,}(Q) \Big \}
\]
with the norm
\[
  \| v \|_{{\mathcal{H}}_{0,}(\Sigma)} := \inf\limits_{V \in
    {\mathcal{H}}_{;0,}(Q): \gamma_\Sigma^i V = v}
  \| V \|_{{\mathcal{H}}(Q)} \, .
\]

\begin{remark}\label{rem:densityTraces}
  By the definition of $\calH_{0,}(\Sigma)$ and using the
  linearity of $\gamma_\Sigma^i$,
  we have that for any $v \in \calH_{0,}(\Sigma)$ there exists a sequence
  $(v_n)_{n\in\N}\subset H^{1/2}_{0,}(\Sigma)$ 
  such that $\lim\limits_{n\to\infty} \norm{v - v_n}{\calH_{0,}(\Sigma)}=0$. 
\end{remark}
  
\begin{remark}
The trace spaces investigated in this paper are closely related to the 
spaces used in  the classical time dependent BEM approach for the wave
equation, introduced by Bamberger and Ha--Duong \cite{BHD86}.
Indeed, as pointed out in \cite[Remark~2]{HAD03},
$H^{1/2}_{0,}(\Sigma)$ agrees with
\begin{align*}
H^{1/2,1/2}_{\sigma,\Gamma}:=\Big \{ u \in LT(\sigma, H^{1/2}(\Gamma))\,;\, 
\int_{\R+i\sigma} \abs{\hat{u}}_{1/2,\omega,\Gamma} d\omega < \infty \Big \}
\end{align*}
when $\sigma = 0$. Additionally, $\left(H^{1/2}_{0,}(\Sigma)\right)^\prime$
corresponds to
\begin{align*}
H^{-1/2,-1/2}_{\sigma,\Gamma}:=\Big \{ u \in LT(\sigma, H^{-1/2}(\Gamma))\,;\, 
\int_{\R+i\sigma} \abs{\hat{u}}_{-1/2,\omega,\Gamma} d\omega < \infty \Big \}
\end{align*}
when $\sigma = 0$. Remarkably, $\sigma$ is taken to be zero for practical
computations and numerical experiments, yet the classical time dependent
BEM does not cover this case. We refer to \cite{HAD03} for the detailed
definitions and a more comprehensive discussion.
\end{remark}

\section{Initial boundary value problems}\label{sec:BVPs}

\subsection{Homogeneous Dirichlet data}
Instead of \eqref{eq:IBVP}, let us first consider the Dirichlet initial
boundary value problem with zero boundary conditions,
\begin{equation}\label{eq:IBVP0}
\begin{array}{rclcl}
\Box u(x,t) & = & f(x,t) & & \text{for} \; (x,t) \in Q, \\
u(x,t) & = & 0 & & \text{for} \; (x,t) \in \Sigma,\\
u(x,0) = \partial_t u(x,t)_{|t=0} & = & 0 & & \text{for} \; x \in \Omega.
\end{array} 
\end{equation}
A possible variational formulation of \eqref{eq:IBVP0} is to find
$u \in H^{1,1}_{0;0,}(Q)$ such that
\begin{equation}\label{eq:IBVP0 VF}
  - \int_0^T \int_\Omega \partial_t u \, \partial_t v \, dx \, dt +
  \int_0^T \int_\Omega \nabla_x u \cdot \nabla_x v \, dx \, dt =
  \int_0^T \int_\Omega f \, v \, dx \, dt
\end{equation}
is satisfied for all $v \in H^{1,1}_{0;,0}(Q)$. When assuming $f \in L^2(Q)$
we are able to construct a unique solution $u \in H^{1,1}_{0;0,}(Q)$ of the
variational formulation \eqref{eq:IBVP0}, satisfying the stability estimate
\cite[Theorem 5.1]{StZ20}, see also \cite[Chapt.~IV, Theorem~3.1]{Lad85},
\[
\| u \|_{H^{1,1}_{0;0,}(Q)} \leq \frac{1}{\sqrt{2}} \, T \, \| f \|_{L^2(Q)} .
\]
While the variational formulation \eqref{eq:IBVP0 VF} is well posed also
for $f \in [H^{1,1}_{0;,0}(Q)]'$, it is not possible to prove a related
inf-sup condition to ensure the existence of a unique solution
$u \in H^{1,1}_{0;0,}(Q)$, see \cite[Theorem 4.2.24]{Zan19}. 
However, by definition we have the inf-sup condition
\begin{equation}\label{inf sup wave}
  \| \Box \widetilde{u} \|_{[H^1_0(Q_-)]'} =
  \sup\limits_{0 \neq v \in H^{1,1}_{0;,0}(Q)}
  \frac{|\langle \Box \widetilde{u} , {\mathcal{E}}v \rangle_{Q_-}|}
  {\| v \|_{H^{1,1}_{;,0}(Q)}} \quad \mbox{for all} \; u \in
  {\mathcal{H}}_{0;0,}(Q),
\end{equation}
and therefore we conclude unique solvability of the variational formulation
to find $u \in {\mathcal{H}}_{0;0,}(Q)$ such that
\begin{equation}\label{eq:gen VF}
  \langle \Box \widetilde{u} , {\mathcal{E}}v \rangle_{Q_-} =
  \langle f , v \rangle_Q
\end{equation}
is satisfied for all $ v \in H^{1,1}_{0;,0}(Q)$, see \cite[Theorem 3.9]{StZ21}.
Moreover, for the solution $u$ it holds
\[
\| \Box \widetilde{u} \|_{[H^1_0(Q_-)]'} =
  \sup\limits_{0 \neq v \in H^{1,1}_{0;,0}(Q)}
  \frac{|\langle \Box \widetilde{u} , {\mathcal{E}}v \rangle_{Q_-}|}
  {\| v \|_{H^{1,1}_{0;,0}(Q)}} =
  \sup\limits_{0 \neq v \in H^{1,1}_{0;,0}(Q)}
  \frac{|\langle f , v \rangle_Q|}
  {\| v \|_{H^{1,1}_{0;,0}(Q)}} \leq \| f \|_{[H^{1,1}_{0;,0}(Q)]'} .
\]
In fact, \eqref{eq:gen VF} is the variational formulation of the operator
equation ${\mathcal{E}}' \Box \widetilde{u} = f$ in
$[H^{1,1}_{0;,0}(Q)]'$, i.e.,
\[
  f_u(v) := \langle {\mathcal{E}}' \Box \widetilde{u} , v \rangle_Q
  = \langle \Box \widetilde{u} , {\mathcal{E}} v \rangle_{Q_-}
  \quad \mbox{for} \; v \in H^{1,1}_{0;,0}(Q) \subset H^{1,1}_{;,0}(Q) 
\]
is a continuous linear functional with norm
\[
  \| f_u \|_{[H^{1,1}_{;,0}(Q)]'} =
  \sup\limits_{0 \neq v \in H^{1,1}_{0;,0}(Q)}
  \frac{|f_u(v)|}{\| v \|_{H^{1,1}_{;,0}(Q)}} =
  \sup\limits_{0 \neq v \in H^{1,1}_{0;,0}(Q)}
  \frac{|\langle \Box \widetilde{u} , {\mathcal{E}} v \rangle_{Q_-}}
  {\| v \|_{H^{1,1}_{;,0}(Q)}} =
  \| \Box \widetilde{u} \|_{[H^1_0(Q_-)]'} .
\]
Recall that for $ u \in H^{1,1}_{0;0,}(Q) \subset {\mathcal{H}}_{0;0,}(Q)$
we have
\[
  f_u(v) = \langle \Box \widetilde{u} , {\mathcal{E}} v \rangle_{Q_-} =
  - \langle \partial_t u , \partial_t v \rangle_{L^2(Q)} +
  \langle \nabla_x u , \nabla_x v \rangle_{L^2(Q)} \quad \mbox{for all} \;
  v \in H^{1,1}_{0;,0}(Q) .
\]
Using the Hahn--Banach theorem, e.g., \cite[Chapt.~IV., Sect.~5]{Yosida:1980},
\cite[Theorem 5.9-1]{Cia13},
there exists a linear continuous functional
$\widetilde{f}_u : H^{1,1,}_{;,0}(Q) \to {\mathbb{R}}$ satisfying
\begin{align}\label{Hahn Banach}
  \widetilde{f}_u(v) = f_u(v) \quad \mbox{for all} \; v \in H^{1,1}_{0;,0}(Q),
  \\
  \| \widetilde{f}_u \|_{[H^{1,1}_{;,0}(Q)]'} =
  \| f_u \|_{[H^{1,1}_{;,0}(Q)]'} =
  \| \Box \widetilde{u} \|_{[H^1_0(Q_-)]'} . \nonumber 
\end{align}
Indeed, for $u\in H^{1,1}_{0;0,}(Q)$, we have the explicit representation
\begin{equation}\label{Representation ftilde}
  \widetilde{f}_u(v) :=
  \langle \Box \widetilde{u} , {\mathcal{E}} v \rangle_{Q_-} =
  - \langle \partial_t u , \partial_t v \rangle_{L^2(Q)} +
  \langle \nabla_x u , \nabla_x v \rangle_{L^2(Q)} \quad \forall \;
  v \in H^{1,1}_{;,0}(Q) .
\end{equation}
In the following we assume $ f \in [H^{1,1}_{;,0}(Q)]'$, and we consider
the variational formulation to find
$\lambda_i \in [H^{1/2}_{,0}(\Sigma)]'$ such that
\begin{equation}\label{VF Def lambda}
  \langle (\gamma_\Sigma^i)' \lambda_i , v \rangle_Q =
  \langle \lambda_i , \gamma_\Sigma^i v \rangle_\Sigma =
  \widetilde{f}_u(v) -
  \langle f , v \rangle_Q \quad \mbox{for all} \; v \in H^{1,1}_{;,0}(Q).
\end{equation}
For $ v \in H^{1,1}_{0;,0}(Q) \subset H^{1,1}_{;,0}(Q)$, it holds
\[
  \widetilde{f}_u(v) - \langle f , v \rangle_Q =
  f_u(v) - \langle f , v \rangle_Q =
  \langle \Box \widetilde{u} , {\mathcal{E}} v \rangle_{Q_-} -
  \langle f , v \rangle_Q = 0,
\]
i.e.,
\[
  \widetilde{f}_u - f \in 
  (\mbox{ker} \, \gamma_\Sigma^i)^0 =
  (H^{1,1}_{0;,0}(Q))^0
  := \Big \{
  g \in [H^{1,1}_{;,0}(Q)]' : \langle g, v \rangle_Q = 0 \quad
  \forall \; v \in H^{1,1}_{0;,0}(Q)
  \Big \} \, .
\]
By the closed range theorem, we obtain
\[
  \widetilde{f}_u - f \in
  \mbox{Im}_{[H^{1/2}_{,0}(\Sigma)]'} (\gamma_\Sigma^i)',
\]
which ensures existence of a solution
$\lambda_i \in [H^{1/2}_{,0}(\Sigma)]'$
of the variational formulation \eqref{VF Def lambda}.
Since the norm in $[H^{1/2}_{,0}(\Sigma)]'$ is defined by duality,
this immediately implies the inf-sup condition
\[
  \| \lambda \|_{[H^{1/2}_{,0}(\Sigma)]'} =
  \sup\limits_{0 \neq v \in H^{1,1}_{;,0}(Q)}
  \frac{|\langle \lambda , \gamma_\Sigma^i v \rangle_\Sigma|}
  {\| v \|_{H^{1,1}_{;,0}(Q)}} \quad
  \mbox{for all} \; \lambda \in [H^{1/2}_{,0}(\Sigma)]',
\]
and therefore uniqueness of $\lambda_i \in [H^{1/2}_{,0}(\Sigma)]'$.
Moreover, this also gives
\begin{align*}
  \| \lambda_i \|_{[H^{1/2}_{,0}(\Sigma)]'} &= 
  \sup\limits_{0 \neq v \in H^{1,1}_{;,0}(Q)}
  \frac{|\langle \lambda_i , \gamma_\Sigma v \rangle_\Sigma|}
        {\| v \|_{H^{1,1}_{;,0}(Q)}} \\
  &= \sup\limits_{0 \neq v \in H^{1,1}_{;,0}(Q)}
        \frac{|\widetilde{f}_u(v) - \langle f , v \rangle_Q|}
      {\| v \|_{H^{1,1}_{;,0}(Q)}} \leq 2 \, \| f \|_{[H^{1,1}_{;,0}(Q)]'},
\end{align*}
where we used
\begin{align*}
  \sup\limits_{0 \neq v \in H^{1,1}_{;,0}(Q)}
  \frac{|\widetilde{f}_u(v)|}{\| v \|_{H^{1,1}_{;,0}(Q)}} &=
  \| \widetilde{f}_u \|_{[H^{1,1}_{;,0}(Q)]'} \\&=
  \| f_u \|_{[H^{1,1}_{;,0}(Q)]'} =
  \| \Box \widetilde{u} \|_{[H^1_0(Q_-)]'} \leq \| f \|_{[H^{1,1}_{;,0}(Q)]'} .
\end{align*}
We now rewrite the variational formulation \eqref{VF Def lambda} as
\[
  \widetilde{f}_u(v) = \langle f , v \rangle_Q +
  \langle \lambda_i , \gamma_\Sigma^i v \rangle_\Sigma , \quad
  v \in H^{1,1}_{;,0}(Q).
\]
In particular, for $u \in H^{1,1}_{0;0,}(Q)$, and using
\eqref{Representation ftilde}, this gives
\[
  - \langle \partial_t u , \partial_t v \rangle_{L^2(Q)} +
  \langle \nabla_x u , \nabla_x v \rangle_{L^2(Q)} =
  \langle f , v \rangle_Q +
  \langle \lambda_i , \gamma_\Sigma^i v \rangle_\Sigma , \quad
  v \in H^{1,1}_{;,0}(Q).
\]
i.e.,
\[
  \Phi(u,v) =
  \langle f , v \rangle_Q + \langle \lambda_i , \gamma_\Sigma^i v \rangle_\Sigma .
\]
When comparing this with Green's first formula \eqref{Green 1 smooth}
for suitable chosen functions, we observe that $\lambda_i$ corresponds
to the spatial normal derivative of $u$.
Hence, also in the general case we shall write
$\gamma_N^i u := \partial_{n_x} u = \lambda_i$
and call this distribution the interior spatial normal derivative of
$ u \in {\mathcal{H}}_{0;0,}(Q)$, i.e.,
\[
\gamma_N^i : {\mathcal{H}}_{0;0,}(Q) \to [H^{1/2}_{,0}(\Sigma)]' .
\]
In a similar way, we also define
\[
\gamma_N^i : {\mathcal{H}}_{0;,0}(Q) \to [H^{1/2}_{0,}(\Sigma)]' .
\]
For a related approach in the case of an elliptic 
equation, see also \cite[pp. 116--117]{McL00}.

\subsection{Inhomogeneous Dirichlet data}
\label{sec:inhomogeneous}
Next we consider the Dirichlet boundary value problem \eqref{eq:IBVP}.
For $g \in {\mathcal{H}}_{0,}(\Sigma)$ there exists, by definition, an
extension $u_g = {\mathcal{E}}_\Sigma g \in {\mathcal{H}}_{;0,}(Q)$, and
the zero extension $\widetilde{u}_g \in L^2(Q_-)$.
Thus, it remains to find
$u_0 := u - u_g \in {\mathcal{H}}_{0;0,}(Q)$ satisfying
\[
  \langle \Box \widetilde{u}_0 , {\mathcal{E}} v \rangle_{Q_-} =
  \langle f , v \rangle_Q -
  \langle \Box \widetilde{u}_g , {\mathcal{E}} v \rangle_{Q_-} 
  \quad \mbox{for all} \;
  v \in H^{1,1}_{0;,0}(Q) .
\]
Note that $u_g \in {\mathcal{H}}_{;0,}(Q) \subset {\mathcal{H}}(Q)$
involves $\Box \widetilde{u}_g \in [H^1_0(Q_-)]'$.
For the solution $u_0$ we obtain
\begin{eqnarray*}
  \| \Box \widetilde{u}_0 \|_{[H^1_0(Q_-)]'}
  & = & \sup\limits_{0 \neq v\in H^{1,1}_{0;,0}(Q)}
  \frac{|\langle \Box \widetilde{u}_0 , {\mathcal{E}} v \rangle_{Q_-}|}
        {\| v \|_{H^{1,1}_{;,0}(Q)}} \, = \,
  \sup\limits_{0 \neq v\in H^{1,1}_{0;,0}(Q)}
        \frac{|\langle f , v \rangle_Q -
        \langle \Box \widetilde{u}_g , {\mathcal{E}} v \rangle_{Q_-}|}
        {\| v \|_{H^{1,1}_{;,0}(Q)}} \\
  & \leq & \| f \|_{[H^{1,1}_{0;,0}(Q)]'} +
        \| {\mathcal{E}} \|_{H^{1,1}_{0;,0}(Q),H^1_0(Q_-)} 
        \| g \|_{{\mathcal{H}}_{0,}(\Sigma)},
\end{eqnarray*}
where we have used
\[
  \| \Box \widetilde{u}_g \|_{[H^1_0(Q_-)]'} \leq
  \sqrt{\| u_g \|_{L^2(Q)}^2 + \| \Box \widetilde{u}_g \|_{[H^1_0(Q_-)]'}^2}
  = \| u_g \|_{{\mathcal{H}}(Q)}
  = \| g \|_{{\mathcal{H}}_{0,}(\Sigma)} .
\]
As before, we can determine $\lambda_i \in [H^{1/2}_{,0}(\Sigma)]'$ as unique
solution of the variational formulation \eqref{VF Def lambda},
and where $\gamma_N^i u := \lambda_i$
is again the spatial normal derivative of the solution $u$ of the
Dirichlet boundary value problem \eqref{eq:IBVP}, satisfying
\begin{equation}\label{eq:lambda f g}
  \| \lambda_i \|_{[H^{1/2}_{,0}(\Sigma)]'} \leq
  2 \, \| f \|_{[H^{1,1}_{0;,0}(Q)]'} +
  \| {\mathcal{E}} \|_{H^{1,1}_{0;,0}(Q),H^1_0(Q_-)} 
  \| g \|_{{\mathcal{H}}_{0,}(\Sigma)} \, .
\end{equation}
Specially, for $ f \equiv 0$, this describes the interior 
Dirichlet to Neumann map $g \mapsto \lambda_i = \gamma_N^i u$, where $u$ is the
solution of the homogeneous wave equation with zero initial data.
This can be written as $ \lambda_i = \OS_i g$, where
$\OS_i : {\mathcal{H}}_{0,}(\Sigma) \to [H^{1/2}_{,0}(\Sigma)]'$ is the
so-called Steklov-Poincar\'e operator, and from \eqref{eq:lambda f g}
we immediately conclude
\begin{equation}\label{SPO boundedness}
  \| \OS_i g \|_{[H^{1/2}_{,0}(\Sigma)]'} \leq c_2^{S_i} \,
  \| g \|_{{\mathcal{H}}_{0,}(\Sigma)} \quad \mbox{for all} \;
  g \in {\mathcal{H}}_{0,}(\Sigma), 
\end{equation}
with $c_2^{S_i} := \| {\mathcal{E}} \|_{H^{1,1}_{0;,0}(Q),H^1_0(Q_-)} $.

As before, we can write the variational formulation \eqref{VF Def lambda} as
\[
  \widetilde{f}_u(v) = \langle f , v \rangle_Q +
  \langle \lambda_i , \gamma_\Sigma^i v \rangle_\Sigma , \quad
  v \in H^{1,1}_{;,0}(Q) .
\]
Now, for $ u \in {\mathcal{H}}_{;0,}(Q)$ there exists a sequence
$ \{ u_n \}_{n \in {\mathbb{N}}} \subset H^{1,1}_{;,0}(Q)$ with
\[
\lim\limits_{n \to \infty} \| u - u_n \|_{{\mathcal{H}}(Q)} = 0 .
\]
Hence we can write
\[
  \widetilde{f}_u(v) = \lim\limits_{n \to \infty} \widetilde{f}_{u_n}(v) =
  \lim\limits_{n \to \infty} \Big[
  - \langle \partial_t u_n , \partial_t v \rangle_{L^2(Q)} +
  \langle \nabla_x u_n , \nabla_x v \rangle_{L^2(Q)} \Big] \, .
\]
In particular, for $v \in H^{1,1}_{;,0}(Q) \cap H^2(Q)$, we can apply integration
by parts to obtain
\begin{eqnarray*}
  \widetilde{f}_u(v)
  & = & \lim\limits_{n \to \infty} \Big[
  \langle u_n , \Box v \rangle_{L^2(Q)} +
  \langle \gamma_\Sigma^i u_n , \gamma_N^i v \rangle_\Sigma -
        \langle u_n(T) , \partial_t v(T) \rangle_{L^2(\Omega)} \Big] \\
  & = & \langle u , \Box v \rangle_{L^2(Q)} +
  \langle \gamma_\Sigma^i u , \gamma_N^i v \rangle_\Sigma -
        \langle u(T) , \partial_t v(T) \rangle_{L^2(\Omega)} \, .
\end{eqnarray*}
With this we finally obtain Green's second formula for the solution
$u \in {\mathcal{H}}_{;0,}(Q)$ of \eqref{eq:IBVP} and
$v \in H^{1,1}_{;,0}(Q) \cap H^2(Q)$,
\begin{equation}\label{Green 2}
\langle u , \Box v \rangle_{L^2(Q)} +
  \langle \gamma_\Sigma^i u , \gamma_N^i v \rangle_\Sigma -
  \langle u(T) , \partial_t v(T) \rangle_{L^2(\Omega)} =
  \langle f , v \rangle_Q +
  \langle \gamma_N^i u , \gamma_\Sigma^i v \rangle_\Sigma .
\end{equation}

\subsection{The Neumann boundary value problem}
We now consider (as in \eqref{VF Def lambda}) the variational problem to find
$ u \in {\mathcal{H}}_{;0,}(Q)$ such that
\begin{equation}\label{NBVP VF}
  \widetilde{f}_u(v)
  =
  \langle f , v \rangle_Q + \langle \lambda , \gamma_\Sigma^i v \rangle_\Sigma
  \quad \mbox{for all} \; v \in H^{1,1}_{;,0}(Q),
\end{equation}
when $\lambda \in [H^{1/2}_{,0}(\Sigma)]'$ is given.
This is the generalized variational formulation of the
Neumann boundary value problem
\begin{equation}\label{eq:Neumann}
\Box u = f \quad \mbox{in} \; Q, \quad
\gamma_N^i u = \lambda \quad \mbox{on} \; \Sigma, \quad
u = \partial_t u = 0 \quad \mbox{on} \; \Sigma_0.
\end{equation}

\begin{lemma}
  For all $ u \in {\mathcal{H}}_{;0,}(Q)$ there holds the inf-sup
  stability condition
  \begin{equation}\label{inf-sup Neumann}
    \frac{\sqrt{2}}{\sqrt{2+T^2}} \,
    \| u \|_{{\mathcal{H}}(Q)} \leq
    \sup\limits_{0 \neq v \in H^{1,1}_{;,0}(Q)}
    \frac{|\widetilde{f}_u(v)|}{\| v \|_{H^{1,1}_{;,0}(Q)}} \, .
  \end{equation}
\end{lemma}
\begin{proof}
  Using \eqref{Hahn Banach} and the norm definition by duality, we first have
  \[
    \| \Box \widetilde{u} \|_{[H^1_0(Q_-)]'} =
    \| \widetilde{f}_u \|_{[H^{1,1}_{;,0}(Q)]'} =
    \sup\limits_{0 \neq v \in H^{1,1}_{;,0}(Q)}
    \frac{|\widetilde{f}_u(v)|}{\| v \|_{H^{1,1}_{;,0}(Q)}} \, .
  \]
  Now, for $0 \neq u \in {\mathcal{H}}_{;0,}(Q)$, there exists a
  non-trivial sequence $(u_n)_{n \in {\mathbb{N}}} \subset H^{1,1}_{;0,}(Q)$,
  $u_n \not\equiv 0$, with
  \[
    \lim\limits_{n \to \infty} \| u - u_n \|_{{\mathcal{H}}(Q)} = 0.
  \]
  For each $u_n \in H^{1,1}_{;0,}(Q)$ we can write, as in
  \eqref{Representation ftilde},
  \[
    \widetilde{f}_{u_n}(v) =
    - \langle \partial_t u_n , \partial_t v \rangle_{L^2(Q)} +
    \langle \nabla_x u_n ,\nabla_x v \rangle_{L^2(Q)} \quad
    \mbox{for all} \; v \in H^{1,1}_{;,0}(Q) ,
  \]
  and we define $w_n \in H^{1,1}_{;,0}(Q)$ as the unique solution of the
  variational formulation
  \[
    - \langle \partial_t v , \partial_t w_n \rangle_{L^2(Q)} +
    \langle \nabla_x v ,\nabla_x w_n \rangle_{L^2(Q)} =
    \langle u_n , v \rangle_{L^2(Q)} \quad \mbox{for all} \;
    v \in H^{1,1}_{;0,}(Q).
  \]
  This variational formulation corresponds to a Neumann boundary value
  problem for the wave equation with a volume source $u_n \in L^2(Q)$,
  and zero conditions at the terminal time $t=T$. 
  As for the related Dirichlet problem we conclude the bound
  \[
    \| w_n \|_{H^{1,1}_{;,0}(Q)} \leq
    \frac{1}{\sqrt{2}} \, T \, \| u_n \|_{L^2(Q)} .
  \]
  In particular, for the test function $v=u_n$, the variational formulation
  gives
  \[
    - \langle \partial_t u_n , \partial_t w_n \rangle_{L^2(Q)} +
    \langle \nabla_x u_n ,\nabla_x w_n \rangle_{L^2(Q)} = \| u_n \|^2_{L^2(Q)}.
  \]
  With this, we now conclude
  \begin{eqnarray*}
    \| \widetilde{f}_{u_n} \|_{[H^{1,1}_{;,0}(Q)]'}
    & = & \sup\limits_{0 \neq v \in H^{1,1}_{;,0}(Q)}
          \frac{|\widetilde{f}_{u_n}(v)|}{\| v \|_{H^{1,1}_{;,0}(Q)}}
          \geq
          \frac{|\widetilde{f}_{u_n}(w_n)|}{\| w_n \|_{H^{1,1}_{;,0}(Q)}} \\
    & & \hspace*{-2.5cm}
        = \frac{|- \langle \partial_t u_n , \partial_t w_n \rangle_{L^2(Q)} +
          \langle \nabla_x u_n ,\nabla_x w_n \rangle_{L^2(Q)}|}
          {\| w_n \|_{H^{1,1}_{;,0}(Q)}} =
          \frac{\| u_n \|^2_{L^2(Q)}}{\| w_n \|_{H^{1,1}_{;,0}(Q)}} \geq
          \frac{\sqrt{2}}{T} \, \| u_n \|_{L^2(Q)} \, .
  \end{eqnarray*}
  Completion for $n \to \infty$ now gives
  \[
    \| \widetilde{f}_u \|_{[H^{1,1}_{;,0}(Q)]'} \geq
          \frac{\sqrt{2}}{T} \, \| u \|_{L^2(Q)} \, .
  \]
  Hence, we can write, for some $\alpha \in (0,1)$,
  \begin{eqnarray*}
    \| \widetilde{f}_u \|^2_{[H^{1,1}_{;,0}(Q)]'}
    & = & \alpha \, \| \widetilde{f}_u \|^2_{[H^{1,1}_{;,0}(Q)]'} +
          (1-\alpha) \, \| \widetilde{f}_u \|^2_{[H^{1,1}_{;,0}(Q)]'} \\
    & \geq & \alpha \, \frac{2}{T^2} \, \| u \|_{L^2(Q)}^2 +
             (1-\alpha) \, \| \Box \widetilde{u} \|^2_{[H^1_0(Q_-)]'} \\
    & = & (1-\overline{\alpha}) \Big( \| u \|_{L^2(Q)}^2 +
          \| \Box \widetilde{u} \|^2_{[H^1_0(Q_-)]'}  \Big) =
          (1-\overline{\alpha}) \, \| u \|^2_{{\mathcal{H}}(Q)},
  \end{eqnarray*}
  when
  \[
    1-\overline{\alpha} = \overline{\alpha} \, \frac{2}{T^2}
  \]
  is satisfied, i.e.,
  \[
    \overline{\alpha} = \frac{T^2}{2+T^2}, \quad
    1-\overline{\alpha} = \frac{2}{2+T^2} .
  \]
  This concludes the proof.
\end{proof}

\begin{lemma}
  For all $ 0 \neq v \in H^{1,1}_{;,0}(Q)$, there exists a function
  $u_v \in {\mathcal{H}}_{;0,}(Q)$ such that
  \begin{equation}\label{injectivity Neumann}
    \widetilde{f}_{u_v}(v) > 0 .
  \end{equation}
\end{lemma}
\begin{proof}
  For $ 0 \neq v \in H^{1,1}_{;,0}(Q)$, there exists a unique solution
  $u_v \in H^{1,1}_{;0,}(Q) \subset {\mathcal{H}}_{;0,}(Q)$, satisfying
  \[
    - \langle \partial_t u_v , \partial_t w \rangle_{L^2(Q)} +
    \langle \nabla_x u_v , \nabla_x w \rangle_{L^2(Q)} =
    \langle v , w \rangle_{L^2(Q)} \quad
    \mbox{for all} \; w \in H^{1,1}_{;,0}(Q),
  \]
  and, for $w=v$, we obtain
  \[
    \widetilde{f}_{u_v}(v) =
    - \langle \partial_t u_v , \partial_t v \rangle_{L^2(Q)} +
    \langle \nabla_x u_v , \nabla_x v \rangle_{L^2(Q)} = \| v \|^2_{L^2(Q)} > 0 .
  \]
\end{proof}

\noindent
The inf-sup condition \eqref{inf-sup Neumann} and the
surjectivity condition \eqref{injectivity Neumann} ensure unique
solvability of the variational formulation \eqref{NBVP VF}, i.e., for
the unique solution $u \in {\mathcal{H}}_{;0,}(Q)$ we obtain
\begin{eqnarray*}
  \frac{\sqrt{2}}{\sqrt{2+T^2}} \, \| u \|_{{\mathcal{H}}(Q)}
  & \leq & \sup\limits_{0 \neq v \in H^{1,1}_{;,0}(Q)}
           \frac{|\widetilde{f}_u(v)|}{\| v \|_{H^{1,1}_{;,0}(Q)}} \\
  & = & \sup\limits_{0 \neq v \in H^{1,1}_{;,0}(Q)}
        \frac{|\langle f , v \rangle_Q +
        \langle \lambda , \gamma_\Sigma^i v \rangle_\Sigma|}
        {\| v \|_{H^{1,1}_{;,0}(Q)}} \, \leq \,
        \| f \|_{[H^{1,1}_{;,0}(Q)]'} +
        \| \lambda \|_{[H^{1/2}_{,0}(\Sigma)]'},
\end{eqnarray*}
and when taking the lateral trace this gives
\begin{equation}\label{eq:u f lambda}
  \| \gamma_\Sigma^i u \|_{{\mathcal{H}}_{0,}(\Sigma)} \leq
  \| u \|_{{\mathcal{H}}(Q)} \leq \frac{1}{\sqrt{2}} \sqrt{2+T^2} \,
  \Big[ \| f \|_{[H^{1,1}_{;,0}(Q)]'} +
  \| \lambda \|_{[H^{1/2}_{,0}(\Sigma)]'} \Big] .
\end{equation}
In particular, for $f \equiv 0$, this defines the interior Neumann to Dirichlet
map $ \lambda \mapsto \gamma_\Sigma^i u$ which can be written as
$\gamma_\Sigma^i u = \OS_i^{-1} \lambda$ when using the inverse of the
Steklov--Poincar\'e operator $\OS_i$. From \eqref{eq:u f lambda}
we then conclude
\begin{equation}\label{ISPO boundedness}
  \| \OS_i^{-1} \lambda \|_{{\mathcal{H}}_{0,}(\Sigma)} \leq
   c_2^{S_i^{-1}} \,
   \| \lambda \|_{[H^{1/2}_{,0}(\Sigma)]'} \quad \mbox{for all} \;
   \lambda \in [H^{1/2}_{,0}(\Sigma)]', \quad
   c_2^{S_i^{-1}} := \frac{1}{\sqrt{2}} \sqrt{2+T^2} \, .
\end{equation}
Now, using \eqref{SPO boundedness}
and duality this gives
\[
  \| \lambda \|_{[H^{1/2}_{,0}(\Sigma)]'} =
  \| \OS_i \gamma_\Sigma^i u \|_{[H^{1/2}_{,0}(\Sigma)]'} \leq
  c_2^{S_i} \, \| \gamma_\Sigma^i u \|_{{\mathcal{H}}_{0,}(\Sigma)} =
  c_2^{S_i} \sup\limits_{0 \neq \mu \in [{\mathcal{H}}_{0,}(\Sigma)]'}
  \frac{|\langle \gamma_\Sigma^i u , \mu \rangle_\Sigma|}
  {\| \mu \|_{[{\mathcal{H}}_{0,}(\Sigma)]'}},
\]
i.e., the inf-sup stability condition
\begin{equation}\label{inf-sup inverse SPO}
  \frac{1}{c_2^{S_i}} \, \| \lambda \|_{[H^{1/2}_{,0}(\Sigma)]'} \leq
  \sup\limits_{0 \neq \mu \in [{\mathcal{H}}_{0,}(\Sigma)]'}
  \frac{|\langle \OS_i^{-1} \lambda , \mu \rangle_\Sigma|}
  {\| \mu \|_{[{\mathcal{H}}_{0,}(\Sigma)]'}} \quad
  \mbox{for all} \; \lambda \in [H^{1/2}_{,0}(\Sigma)]' .
\end{equation}
Furthermore, using \eqref{ISPO boundedness} for $g_i := \gamma_\Sigma^i u$
and duality we obtain
\begin{align*}
  \| g_i \|_{{\mathcal{H}}_{0,}(\Sigma)} &=
  \| \gamma_\Sigma^i u \|_{{\mathcal{H}}_{0,}(\Sigma)} =
  \| \OS_i^{-1} \lambda \|_{{\mathcal{H}}_{0,}(\Sigma)} \\
  &\leq
  c_2^{S_i^{-1}} \, \| \lambda \|_{[H^{1/2}_{,0}(\Sigma)]'} =
  c_2^{S_i^{-1}} \,
  \sup\limits_{0 \neq v \in H^{1/2}_{,0}(\Sigma)}
  \frac{|\langle \lambda , v \rangle_\Sigma|}{\| v \|_{H^{1/2}_{,0}(\Sigma)}},
\end{align*}
i.e., the inf-sup condition
\begin{equation}\label{inf-sup SPO}
  \frac{1}{c_2^{S_i^{-1}}} \,
  \| g_i \|_{{\mathcal{H}}_{0,}(\Sigma)} \leq
   \sup\limits_{0 \neq v \in H^{1/2}_{,0}(\Sigma)}
   \frac{|\langle \OS_i g_i , v \rangle_\Sigma|}{\| v \|_{H^{1/2}_{,0}(\Sigma)}}
   \quad \mbox{for all} \; g_i \in {\mathcal{H}}_{0,}(\Sigma) .
\end{equation}

\subsection{Adjoint problems}
Related to the variational problem \eqref{NBVP VF} we now
consider the adjoint problem to find
$w\in H^{1,1}_{;,0}(Q)$ such that
\begin{equation}\label{adjoint VF}
  \widetilde{f}_u(w) =
  \langle f , u \rangle_Q + \langle g , \gamma_\Sigma^i u \rangle_\Sigma
\end{equation}
is satisfied for all $u \in {\mathcal{H}}_{;0,}(Q)$.
For $ w \in H^{1,1}_{;,0}(Q)$, let $u_w \in {\mathcal{H}}_{;0,}(Q)$
be the unique solution of the variational problem
\[
  \widetilde{f}_{u_w}(v) =
  \langle \partial_t w , \partial_t v \rangle_{L^2(Q)} +
  \langle \nabla_x w , \nabla_x v \rangle_{L^2(Q)} \quad
  \mbox{for all} \; v \in H^{1,1}_{;,0}(Q) .
\]
For $v=w$, this gives
\[
\widetilde{f}_{u_w}(w) = \| w \|^2_{H^{1,1}_{;,0}(Q)} \, .
\]
Moreover, using the inf-sup stability condition
\eqref{inf-sup Neumann}, we obtain
\begin{eqnarray*}
  \frac{\sqrt{2}}{\sqrt{2+T^2}} \, \| u_w \|_{{\mathcal{H}}(Q)}
  & \leq & \sup\limits_{0 \neq v \in H^{1,1}_{;,0}(Q)}
           \frac{|\widetilde{f}_{u_w}(v)|}{\| v \|_{H^{1,1}_{;,0}(Q)}} \\
  & = & \sup\limits_{0 \neq v \in H^{1,1}_{;,0}(Q)}
        \frac{|\langle \partial_t w , \partial_t v \rangle_{L^2(Q)} +
        \langle \nabla_x w , \nabla_x v \rangle_{L^2(Q)}|}
        {\| v \|_{H^{1,1}_{;,0}(Q)}} \, \leq \, \| w \|_{H^{1,1}_{;,0}(Q)},
\end{eqnarray*}
and thus it follows that
\[
  \widetilde{f}_{u_w}(w) = \| w \|^2_{H^{1,1}_{;,0}(Q)} \geq
  \frac{\sqrt{2}}{\sqrt{2+T^2}} \, \| u_w \|_{{\mathcal{H}}(Q)}
  \| w \|_{H^{1,1}_{;,0}(Q)} .
\]
In other words, we have
\[
  \frac{\sqrt{2}}{\sqrt{2+T^2}} \,
  \| w \|_{H^{1,1}_{;,0}(Q)} \leq
  \sup\limits_{0 \neq u \in {\mathcal{H}}_{;0,}(Q)}
  \frac{|\widetilde{f}_u(w)|}
  {\| u \|_{{\mathcal{H}}(Q)}} \quad
  \mbox{for all} \; w \in H^{1,1}_{;,0}(Q) .
\]
Since the inf-sup condition \eqref{inf-sup Neumann} also
implies surjectivity, unique solvability of the variational formulation
\eqref{adjoint VF} follows. In fact, for $ f \in [{\mathcal{H}}_{;0,}(Q)]'$
and $g \in [{\mathcal{H}}_{0,}(\Sigma)]'$ we have $w \in H^{1,1}_{;,0}(Q)$ as the
weak solution of the adjoint Neumann problem for the wave equation
\begin{equation}\label{adjoint wave Neumann}
  \Box w = f \quad \mbox{in} \; Q, \quad
  \gamma_N^i w = g \quad \mbox{on} \; \Sigma, \quad
  w = \partial_t w = 0 \quad \mbox{on} \; \Sigma_T.
\end{equation}

\section{Boundary Integral Equations}\label{sec:BIE}

\subsection{Representation formula}
\label{ssec:FsRfRp}
Let $u \in {\mathcal{H}}_{;0,}(Q)$ be a solution of the generalized
wave equation ${\mathcal{E}}' \Box u = f$ in $[H^{1,1}_{;,0}(Q)]'$ .
For $(x,t) \in Q$  and
$v(y,\tau)=\kappa_tG(x-y,\tau)$, with $G(\cdot,\cdot)$ being the fundamental
solution introduced in \eqref{eq:fundsol} and $\kappa_t$ the 
time-reversal map from \eqref{eq:kt}, formula \eqref{Green 2}
becomes the representation formula
\begin{align*}
  u(x,t) = \int_0^t \int_\Omega f(y,\tau) \,  G(x-y,t-\tau)
  \, dy \, d\tau + \dual{\gamma_N^i u}{\gamma_{\Sigma}^i v}_{\Sigma} -
  \dual{\gamma_{\Sigma}^i u}{\gamma_N^i v}_{\Sigma}. 
\end{align*}
In particular, for $f \equiv 0$, we conclude the following
representation formula 
\begin{align}
u(x,t)  \label{eq:repf}
= (\mathscr{S}\gamma_N^i u)(x,t) - (\mathscr{D}\gamma_{\Sigma}^i  u)(x,t),
\quad (x,t) \in Q,
\end{align}
with the single and double layer potentials $\mathscr{S}$ and
$\mathscr{D}$, defined as in \eqref{eq:SL} and \eqref{eq:DL}, respectively.

\subsection{Single layer potential}
We first recall the definition \eqref{eq:SL} of the single layer potential
\[
  u_w(x,t) = ({\mathscr{S}} w)(x,t) =
  \int_0^t \int_\Gamma G(x-y,t-\tau) \, w(y,\tau) \, ds_y \, d\tau , \quad
  (x,t) \in Q .
\]

\begin{proposition}\label{prop:Scont}
  For the single layer potential we have
  \[
    {\mathscr{S}} : [H^{1/2}_{,0}(\Sigma)]' \to {\mathcal{H}}_{;0,}(Q) .
  \]
\end{proposition}
\begin{proof}
  For $u_w={\mathscr{S}}w$ and a suitable $\psi$, we can write the duality
  pairing as extension of the inner product in $L^2(Q)$ as
\begin{eqnarray*}
  \langle u_w , \psi \rangle_Q
  & = & \int_0^T \int_\Omega u_w(x,t) \, \psi(x,t) \, dx \, dt \\
  & = & \int_0^T \int_\Omega \int_0^t \int_\Gamma G(x-y,t-\tau) \,
        w(y,\tau) \, ds_y \, d\tau \, \psi(x,t) \, dx \, dt \\
  & = & \int_0^T \int_\Gamma w(y,\tau)
        \int_\tau^T G(x-y,t-\tau) \, \psi(x,t) \, dx \, dt \, ds_y \, d\tau \\
  & = & \int_0^T \int_\Gamma w(y,\tau) \,
        \varphi_\psi(y,\tau) \, ds_y \, d\tau \\[2mm]
  & = & \langle w , \gamma_\Sigma^i \varphi_\psi \rangle_\Sigma,
\end{eqnarray*}
where
\[
  \varphi_\psi(y,\tau) =
  \int_\tau^T G(x-y,t-\tau) \, \psi(x,t) \, dx \, dt, \quad (y,\tau) \in Q,
\]
is a solution of the adjoint problem
\eqref{adjoint wave Neumann}. Hence, for
$\psi \in [{\mathcal{H}}_{;0,}(Q)]'$, we obtain
$\varphi_\psi \in H^{1,1}_{;,0}(Q)$, and therefore
$\gamma_\Sigma^i \varphi_\psi \in H^{1/2}_{,0}(\Sigma)$.
From this, we conclude that
$ u_w \in {\mathcal{H}}_{;0,}(Q)$ when
$w \in [H^{1/2}_{,0}(\Sigma)]'$ is given.
\end{proof}

\noindent
As a corollary of the previous result, we can define the single layer
boundary integral operator
\begin{equation}\label{Def V}
  \OV := \gamma_\Sigma^i {\mathscr{S}} : [H^{1/2}_{,0}(\Sigma)]' \to
  {\mathcal{H}}_{0,}(\Sigma) ,
\end{equation}
and the normal derivative of the single layer potential,
\begin{equation}\label{Def V deriv}
  \gamma_N^i {\mathscr{S}} : [H^{1/2}_{,0}(\Sigma)]' \to
  [H^{1/2}_{,0}(\Sigma)]'.
\end{equation}

\subsection{Double layer potential} We first recall the definition
\eqref{eq:DL} of the double layer potential
\[
  u_z(x,t) = ({\mathscr{D}} z)(x,t) =
  \int_0^t \int_\Gamma \partial_{n_y} G(x-y,t-\tau) \,
  z(y,\tau) \, ds_y \, d\tau , \quad (x,t) \in Q .
\]

\begin{proposition}\label{prop:Dcont}
  For the double layer potential we have
  \[
    {\mathscr{D}} : {\mathcal{H}}_{0,}(\Sigma) \to {\mathcal{H}}_{;0,}(Q) .
  \]
\end{proposition}
\begin{proof}
  For $u_z={\mathscr{D}}z $ and a suitable $\psi$ we can write the duality
  pairing as extension of the inner product in $L^2(Q)$ as
\begin{eqnarray*}
  \langle u_z , \psi \rangle_Q
  & = & \int_0^T \int_\Omega u_z(x,t) \, \psi(x,t) \, dx \, dt \\
  & = & \int_0^T \int_\Omega \int_0^t \int_\Gamma \partial_{n_y} G(x-y,t-\tau) \,
        z(y,\tau) \, ds_y \, d\tau \, \psi(x,t) \, dx \, dt \\
  & = & \int_0^T \int_\Gamma z(y,\tau) \; \partial_{n_y}
        \int_\tau^T G(x-y,t-\tau) \, \psi(x,t) \, dx \, dt \, ds_y \, d\tau \\
  & = & \int_0^T \int_\Gamma z(y,\tau) \,
        \partial_{n_y} \varphi_\psi(y,\tau) \, ds_y \, d\tau \\[2mm]
  & = & \langle z , \gamma_N^i \varphi_\psi \rangle_\Sigma,
\end{eqnarray*}
where
\[
  \varphi_\psi(y,\tau) = \int_\tau^T \int_\Gamma
  G(x-y,t-\tau) \, \psi(x,t) \, dx \, dt, \quad (y,\tau) \in Q,
\]
is a solution of the adjoint problem
\eqref{adjoint wave Neumann}. Hence, for
$\psi \in [{\mathcal{H}}_{;0,}(Q)]'$ we obtain
$\varphi_\psi \in H^{1,1}_{;,0}(Q)$, and 
$\gamma_N^i \varphi_\psi \in [{\mathcal{H}}_{0,}(\Sigma)]'$.
From this we conclude
$ u_z \in {\mathcal{H}}_{;0,}(Q)$ when
$z \in {\mathcal{H}}_{0,}(\Sigma)$ is given.
\end{proof}

\noindent
With the previous result we are in a position to consider the lateral trace
of the double layer potential
\begin{equation}\label{Def D deriv}
  \gamma_\Sigma^i {\mathscr{D}} : {\mathcal{H}}_{0,}(\Sigma) \to
  {\mathcal{H}}_{0,}(\Sigma),
\end{equation}
and the so-called hypersingular boundary integral operator as normal derivative
of the double layer potential,
\begin{equation}\label{Def W}
  \OW := - \gamma_N^i {\mathscr{D}} : {\mathcal{H}}_{0,}(\Sigma) \to
  [H^{1/2}_{,0}(\Sigma)]' .
\end{equation}

\subsection{Boundary integral operators and Calder\'on identities}
\label{ssec:Bios}
Without loss of generality, let us consider the complementary domains
\begin{equation*}\label{eq:OmegacQc}
  \Omega^c := B_R\setminus\overline{\Omega} \quad
  \text{ and } \quad Q^c:=\Omega^c \times (0,T),
\end{equation*}
with $B_R := \lbrace \Vx \in \R^n \, : \, \abs{\Vx}<R\rbrace$ is a sufficiently 
large ball containing $\Gamma$.
With this, we define the exterior traces $\gamma_\Sigma^e$ and $\gamma_N^e$ 
following the same ideas from Subsection~\ref{ssec:TracesQ}, but using $Q^c$ 
instead of $Q$.

\begin{remark}
  Clearly, the mappings
  \begin{align*}
    \gamma_\Sigma^e \,
    &: H^{1,1}_{;0,}(Q^c) \to H^{1/2}_{0,}(\Sigma), \qquad \qquad
      \gamma_\Sigma^e \, : H^{1,1}_{;,0}(Q^c) \to H^{1/2}_{,0}(\Sigma), \\
    \gamma_\Sigma^e \,
    &: \calH_{;0,}(Q^c) \to \calH_{0,}(\Sigma), \qquad \:\,\qquad
      \gamma_\Sigma^e \, : \calH_{;,0}(Q^c) \to \calH_{,0}(\Sigma),
  \end{align*}
  are continuous and surjective, while 
  \begin{align*}
    \gamma_N^e \,
    &: H^{1,1}_{;0,}(Q^c)\to [ H^{1/2}_{,0}(\Sigma)]', 
      \qquad \qquad
      \gamma_N^e \, : H^{1,1}_{;,0}(Q^c)\to [ H^{1/2}_{0,}(\Sigma)]',\\
    \gamma_N^e \,
    &: \calH_{;0,}(Q^c) \to [ H^{1/2}_{,0}(\Sigma)]', 
      \qquad \:\qquad
      \gamma_N^e \, : \calH_{;,0}(Q^c) \to [ H^{1/2}_{0,}(\Sigma)]',
  \end{align*}
  are continuous. Moreover, Green's formulae and other properties of
  the interior trace operators $\gamma_\Sigma^i$ and $\gamma_N^i$ 
  also apply to these exterior traces in their corresponding spaces.
  Indeed, following Propositions \ref{prop:Scont} and 
  \ref{prop:Dcont}, we have the continuity of the mappings
  \[
    \mathscr{S} : [ H^{1/2}_{,0}(\Sigma)]' \to
    \calH_{;0,}(Q^c), \qquad
  \mathscr{D} : \calH_{0,}(\Sigma) \quad \to \calH_{;0,}(Q^c). 
  \]
\end{remark}

\noindent
We define the jumps across $\Sigma$ by 
\[
 \left[ \gamma_\Sigma u \right] := \gamma_\Sigma^e u - \gamma_\Sigma^i u,\qquad
 \left[ \gamma_N u \right] := \gamma_N^e u - \gamma_N^i u,
\]
which clearly do not depend on the choice of $B_R$.
Now we can state the following result:

\begin{proposition}
  The following jump relations hold for all
  $w \in [ H^{1/2}_{,0}(\Sigma)]'$ and $z \in \calH_{0,}(\Sigma)$,
  \[
    [ \gamma_\Sigma \mathscr{S} w ] = 0, \qquad  
    [ \gamma_N \mathscr{S} w ] = -w, \qquad
    [ \gamma_\Sigma \mathscr{D} z ] = z, \qquad  
    [ \gamma_N \mathscr{D} z ] = 0.
  \]
\end{proposition}
\begin{proof}
  The jump relations are known to hold when $w$ and $z$ are smooth, e.g.,
  \cite[Sect.~2.2.1]{CoS17}, and \cite[Sect.~1.3]{Say16}. We extend them to 
  $(w, z) \in [ H^{1/2}_{,0}(\Sigma)]' \times \calH_{0,}(\Sigma)$ by using
  that the combined trace map
  $(\gamma_\Sigma,\gamma_N): u \mapsto (\gamma_\Sigma u,\gamma_N u)$ maps
  $\left.C^\infty_0(\R^n\times\R_+)\right|_{\overline{Q}}$ onto a dense
  subspace of $[ H^{1/2}_{,0}(\Sigma)]' \times H^{1/2}_{0,}(\Sigma)$
  (cf. \cite[Lemma~3.5]{COS88}), and that $H^{1/2}_{0,}(\Sigma)$ is
  dense in $\calH_{0,}(\Sigma)$.
\end{proof}

\noindent
We can now define the  boundary integral operators as follows:

\begin{definition}\label{def:BIOs}
 \begin{align*}
 \OV w &:= \gamma_{\Sigma}^i \mathscr{S} w 
 = \gamma_{\Sigma}^e \mathscr{S} w, \\
\OK z &:= \frac{1}{2}\left( \gamma_{\Sigma}^i \mathscr{D} z
 + \gamma_{\Sigma}^e \mathscr{D} z \right),\\
 \OK' w &:= \frac{1}{2}\left( \gamma_N^i \mathscr{S} w
 + \gamma_N^e \mathscr{S} w \right), \\
  \OW z &:= -\gamma_N^i \mathscr{D} z
  = -\gamma_N^e \mathscr{D} z \, .
\end{align*}
\end{definition}

\noindent
From this definition and \eqref{Def V}, \eqref{Def V deriv},
\eqref{Def D deriv}, \eqref{Def W}, we obtain:

\begin{theorem}\label{thm:BIOsCont}
 The boundary integral operators introduced in Definition~\ref{def:BIOs} are 
 continuous in the following spaces:
 \begin{align*}
\OV &: [ H^{1/2}_{,0}(\Sigma)]' \to \:\,\calH_{0,}(\Sigma),\\
\OK &: \:\,\calH_{0,}(\Sigma)\quad\to \:\,\calH_{0,}(\Sigma),\\
\OK' &: [ H^{1/2}_{,0}(\Sigma)]' \to [ H^{1/2}_{,0}(\Sigma)]', \\
\OW &: \:\,\calH_{0,}(\Sigma)\quad\to [ H^{1/2}_{,0}(\Sigma)]'.
 \end{align*}
\end{theorem}

\noindent
Next, we take traces on the representation formula \eqref{eq:repf} and 
get 
\begin{align*}
  \gamma_{\Sigma}^i u &= (\frac{1}{2} \OI - \OK)\gamma_{\Sigma}^i u  +
                        \OV\gamma_N^i u ,\\
  \gamma_{N}^i u &= \OW \gamma_{\Sigma}^i u + (\frac{1}{2} \OI +
                   \OK') \gamma_N^i u .
\end{align*}
As usual, we can rewrite this as
\begin{align}\label{eq:calderon}
 \left(\begin{array}{c}
  \gamma_{\Sigma}^i u\\ \gamma_{N}^i u
 \end{array}\right) =  \underbrace{\left(\begin{array}{c c }
  (\frac{1}{2} \OI - \OK) & \OV\\ \OW & (\frac{1}{2} \OI + \OK^\prime)
 \end{array}\right)}_{=:\OC_{Q}^i}  \left(\begin{array}{c}
  \gamma_{\Sigma}^i u\\ \gamma_{N}^i u
 \end{array}\right)
\end{align}
with the interior Calderon projection $\OC_Q^i$.

Using standard arguments (see for example \cite[Sect.~1.4]{Say16}), we 
can now prove
\begin{align*}
 \left(\begin{array}{c} z \\ w \end{array}\right) =  
 \OC_Q^i \left(\begin{array}{c} z\\ w \end{array}\right), \quad 
 \forall w \in [ H^{1/2}_{,0}(\Sigma) ]',\, z \in 
\calH_{0,}(\Sigma).
\end{align*}
Furthermore, this gives $(\OC_Q^i)^2 = \OC_Q^i$, from which we get
\[
  \OV \OW = (\frac{1}{2} \OI - \OK)(\frac{1}{2} \OI + \OK), \quad
  \OW \OV = (\frac{1}{2} \OI - \OK^\prime)(\frac{1}{2} \OI + \OK^\prime),
  \quad\OV \OK^\prime = \OK \OV, \quad \OK^\prime \OW = \OW \OK. 
\]

\subsection{Coercivity of boundary integral operators}
In this subsection, we are going to prove coercivity properties of boundary
integral
operators, i.e., of the single layer boundary integral operator $\OV$ and
the hypersingular boundary integral operator $\OW$, which ensure unique
solvability of related boundary integral equations.

\begin{theorem}\label{Thm inf sup V}
  The single layer boundary integral operator
  $\OV : [H^{1/2}_{,0}(\Sigma)]' \to {\mathcal{H}}_{0,}(\Sigma)$
  satisfies the inf-sup stability condition
  \begin{equation}\label{inf sup V}
    c_1^V \, \| w \|_{[H^{1/2}_{,0}(\Sigma)]'} \leq 
    \sup\limits_{0 \neq \mu \in [{\mathcal{H}}_{0,}(\Sigma)]'}
    \frac{|\langle \OV w , \mu \rangle_\Sigma|}
    {\| \mu \|_{[{\mathcal{H}}_{0,}(\Sigma)]'}} \quad
    \mbox{for all} \; w \in [H^{1/2}_{,0}(\Sigma)]' .
  \end{equation}
\end{theorem}
\begin{proof}
For $w \in [H^{1/2}_{,0}(\Sigma)]'$ we consider the single layer potential
$ u = {\mathscr{S}}w$ which defines a solution $u \in {\mathcal{H}}_{0,}(Q)$
of the homogeneous wave equation. When taking the lateral trace of $u$ this
gives $g = \gamma_\Sigma^i u = \OV w \in {\mathcal{H}}_{0,}(\Sigma)$.
In fact, $u$ is the unique solution of the Dirichlet boundary value
problem
\[
  \Box u = 0 \quad \mbox{in} \; Q, \quad
  \gamma_\Sigma^i u = g \quad \mbox{on} \; \Sigma, \quad
  u = \partial_t u = 0 \quad \mbox{on} \; \Sigma_0 .
\]
When using the interior Steklov--Poincar\'e operator $\OS_i$ we can
determine the related interior Neumann trace
\[
\lambda_i = \gamma_N^i u = \OS_i g \in [H^{1/2}_{,0}(\Sigma)]' .
\]
Since the Steklov--Poincar\'e operator $\OS_i$ is invertible, this gives
$g = \OS_i^{-1} \lambda_i$, i.e., $ g = \gamma_\Sigma^i u$ is the lateral trace
of the solution of the Neumann boundary value problem
\[
  \Box u = 0 \quad \mbox{in} \; Q, \quad
  \gamma_N^i u = \lambda_i \quad \mbox{on} \; \Sigma, \quad
  u = \partial_t u = 0 \quad \mbox{on} \; \Sigma_0 .
\]
From the inf-sup stability condition \eqref{inf-sup inverse SPO}
of the inverse interior Steklov-Poincar\'e operator $\OS_i^{-1}$
we now conclude
\[
  \frac{1}{c_2^{S_i}} \, \| \lambda_i \|_{[H^{1/2}_{,0}(\Sigma)]'} \leq
  \sup\limits_{0 \neq \mu \in [{\mathcal{H}}_{0,}(\Sigma)]'}
  \frac{|\langle \OS_i^{-1} \lambda_i , \mu \rangle_\Sigma|}
  {\| \mu \|_{[{\mathcal{H}}_{0,}(\Sigma)]'}} =
  \sup\limits_{0 \neq \mu \in [{\mathcal{H}}_{0,}(\Sigma)]'}
  \frac{|\langle \OV w , \mu \rangle_\Sigma|}
  {\| \mu \|_{[{\mathcal{H}}_{0,}(\Sigma)]'}}
\]
For the exterior problem we can derive a related estimate, i.e.,
\[
  \frac{1}{c_2^{S_e}} \, \| \lambda_e \|_{[H^{1/2}_{,0}(\Sigma)]'} \leq
  \sup\limits_{0 \neq \mu \in [{\mathcal{H}}_{0,}(\Sigma)]'}
  \frac{|\langle \OV w , \mu \rangle_\Sigma|}
  {\| \mu \|_{[{\mathcal{H}}_{0,}(\Sigma)]'}} \, ,
\]
where $\lambda_e$ is the exterior Neumann trace of the single layer
potential $u = {\mathscr{S}}w$. Now, and using the jump relation of the
adjoint double layer potential, this gives
\begin{eqnarray*}
  \| w \|_{[H^{1/2}_{,0}(\Sigma)]'}
  & = & \| \lambda_i - \lambda_e \|_{[H^{1/2}_{,0}(\Sigma)]'} \\
  & \leq &
  \| \lambda_i \|_{[H^{1/2}_{,0}(\Sigma)]'} +
  \| \lambda_e \|_{[H^{1/2}_{,0}(\Sigma)]'}  \leq (c_2^{S_i} + c_2^{S_e})
  \sup\limits_{0 \neq \mu \in [{\mathcal{H}}_{0,}(\Sigma)]'}
  \frac{|\langle \OV w , \mu \rangle_\Sigma|}
  {\| \mu \|_{[{\mathcal{H}}_{0,}(\Sigma)]'}} \, ,
\end{eqnarray*}
which implies the desired inf-sup condition.
\end{proof}

\noindent
While the inf-sup stability condition \eqref{inf sup V} ensures uniqueness of
a solution of a related boundary integral equation, the following
result will provide solvability.

\begin{lemma}\label{Lem >0 V}
  For any $ 0 \neq \mu \in [{\mathcal{H}}_{0,}(\Sigma)]'$ there exists a
  $w_\mu \in [H^{1/2}_{,0}(\Sigma)]'$ such that
  \[
    \langle \OV w_\mu , \mu \rangle_\Sigma > 0 
  \]
  is satisfied.
\end{lemma}
\begin{proof}
  For given $0 \neq \mu \in [{\mathcal{H}}_{0,}(\Sigma)]'$ we define the
  adjoint single layer potential $u_\mu \in H^{1,1}_{;,0}(Q)$ by
  \[
    u_\mu(y,\tau) = \int_\tau^T \int_\Gamma G(x-y,t-\tau) \,
    \mu(x,t) \, ds_x \, dt \quad \mbox{for} \; (y,\tau) \in Q .
  \]
  For the lateral trace $\gamma_\Sigma^i u_\mu \in H^{1/2}_{,0}(\Sigma)$ and
  arbitrary $w \in [H^{1/2}_{,0}(\Sigma)]'$ we then have
  \begin{eqnarray*}
    \langle w , \gamma_\Sigma^i u_\mu \rangle_\Sigma
    & = & \int_0^T \int_\Gamma w(y,\tau)
          \int_\tau^T \int_\Gamma
          G(x-y,t-\tau) \, \mu(x,t) \, ds_x \, dt \, ds_y \, d\tau \\
    & = & \int_0^T \int_\Gamma \int_0^t \int_\Gamma
          G(x-y,t-\tau) \, w(y,\tau) \, ds_y \, d\tau \, \mu(x,t) \, ds_x \, dt
          \, = \, \langle \OV w , \mu \rangle_\Sigma \, .
  \end{eqnarray*}
  Moreover, we compute
  \[
    U_\mu(x,t) := \int_0^t u_\mu(x,s) \, ds \quad \mbox{for} \;
    (x,t) \in Q ,
  \]
  with the lateral trace
  $ g_\mu := \gamma_\Sigma^i U_\mu \in H^{1/2}_{0,}(\Sigma) \subset
  {\mathcal{H}}_{0,}(\Sigma)$. Hence, there exists a unique solution
  $v_\mu \in {\mathcal{H}}_{;0,}(Q)$ of the Dirichlet problem for
  the wave equation,
  \[
    \Box v_\mu = 0 \quad \mbox{in} \; Q, \quad
    v_\mu = g_\mu \quad \mbox{on} \; \Sigma, \quad
    v_\mu = \partial_t v_\mu = 0 \quad \mbox{on} \; \Sigma_0 .
  \]
  We then conclude
  \begin{eqnarray*}
    && \int_0^t \int_\Gamma \frac{\partial}{\partial n_x} v_\mu(x,s) \,
       \partial_s v_\mu(x,s) \, ds_x \, ds \\
    && \hspace*{1cm}
       = \int_0^t \int_\Omega \Big[ \partial_{ss} v_\mu(x,s) \,
       \partial_s v_\mu(x,s) + \nabla_x v_\mu(x,s) \cdot \nabla_x
       \partial_s v_\mu(x,s) \Big] \, dx \, ds \\
    && \hspace*{1cm} = \frac{1}{2} \int_0^t \frac{d}{ds} \int_\Omega
       \Big[ [\partial_s v_\mu(x,s)]^2
       + [\nabla_x v_\mu(x,s)]^2 \Big] \, dx \, ds       \\
    && \hspace*{1cm} = \frac{1}{2} \, \| \partial_t v_\mu(t) \|^2_{L^2(\Omega)}
       + \frac{1}{2} \, \| \nabla_x v_\mu(t) \|^2_{L^2(\Omega)} \geq 0
       \quad \mbox{for all} \; t \in (0,T].
  \end{eqnarray*}
  In the case
  \[
    \frac{1}{2} \, \| \partial_t v_\mu(t) \|^2_{L^2(\Omega)}
    + \frac{1}{2} \, \| \nabla_x v_\mu(t) \|^2_{L^2(\Omega)} = 0
       \quad \mbox{for all} \; t \in (0,T],
   \]
   and together with the zero initial conditions, we would conclude
   $v_\mu \equiv 0$ in $Q$, which then implies $g_\mu \equiv 0$ on
   $ \Sigma$, and thus
   $u_\mu \equiv 0$. But this contradicts $\mu \not \equiv 0$. Therefore
   we have
   \[
     \int_0^T \int_\Gamma \frac{\partial}{\partial n_x} v_\mu(x,t) \,
     \partial_t v_\mu(x,t) \, ds_x \, dt =
      \frac{1}{2} \, \| \partial_t v_\mu(T) \|^2_{L^2(\Omega)}
    + \frac{1}{2} \, \| \nabla_x v_\mu(T) \|^2_{L^2(\Omega)} > 0,
  \]
  and with
  \[
    \partial_t U_\mu = u_\mu \quad \mbox{in} \; Q, \quad
    \partial_t v_\mu = \partial_t g_\mu \quad \mbox{on} \; \Sigma, \quad
    g_\mu = \gamma_\Sigma^i U_\mu, \quad
    w_\mu := \gamma_N^i v_\mu \in [H^{1/2}_{,0}(\Sigma)]'
  \]
  we finally conclude
  \[
    \langle \OV w_\mu , \mu \rangle_\Sigma =
    \langle w_\mu , \gamma_\Sigma^i u_\mu \rangle_\Sigma =
    \int_0^T \int_\Gamma \frac{\partial}{\partial n_x} v_\mu(x,t) \,
     \partial_t v_\mu(x,t) \, ds_x dt > 0 .
  \]
\end{proof}

\noindent
The solution of the Dirichlet boundary value problem
\[
  \Box u = 0 \quad \mbox{in} \; Q, \quad
  u = g \quad \mbox{on} \; \Sigma, \quad
  u = \partial_t u = 0 \quad \mbox{on} \; \Sigma_0
\]
is given by the representation formula
\[
  u(x,t) = ({\mathscr{S}}\gamma_N^i u)(x,t) -
  ({\mathscr{D}}g)(x,t) \quad \mbox{for} \; (x,t) \in Q,
\]
where we can determine the yet unknown Neumann datum
$ w = \gamma_N^i u \in [H^{1/2}_{,0}(\Sigma)]'$ as the unique solution
of the first kind boundary integral equation
\begin{equation}\label{BIE V}
\OV w = (\frac{1}{2}\OI+\OK)g \quad \mbox{on} \; \Sigma ,
\end{equation}
i.e., of the variational formulation
\begin{equation}\label{BIE V VF}
  \langle \OV w , \mu \rangle_\Sigma =
  \langle (\frac{1}{2}\OI+\OK)g,\mu \rangle_\Sigma \quad
  \mbox{for all} \; \mu \in [{\mathcal{H}}_{0,}(\Sigma)]' .
\end{equation}
Solvability of the variational formulation \eqref{BIE V VF} follows from
Lemma \ref{Lem >0 V}, while uniqueness of the solution is a consequence
of Theorem \ref{Thm inf sup V}. Instead of the variational formulation
\eqref{BIE V VF}, we may use the modified Hilbert transformation
${\mathcal{H}}_T$ as defined in subsection~\ref{ssec:HilbertTh}
to end up with an equivalent variational problem to find
$w \in [H^{1/2}_{,0}(\Sigma)]'$ such that
\begin{equation}\label{BIE V H}
  \langle \calH_T \OV w , \mu \rangle_\Sigma =
  \langle \calH_T (\frac{1}{2}\OI+\OK)g, \mu \rangle_\Sigma \quad
  \mbox{for all} \; \mu \in [{\mathcal{H}}_{,0}(\Sigma)]' .
\end{equation}
Due to the inclusion $H^{1/2}_{,0}(\Sigma) \subset {\mathcal{H}}_{,0}(\Sigma)$,
we obviously have
$[{\mathcal{H}}_{,0}(\Sigma)]' \subset [H^{1/2}_{,0}(\Sigma)]'$ which will
allow for a Galerkin--Bubnov space-time boundary element discretization
of \eqref{BIE V H}.

\begin{remark}
For a solution $u$ of the homogeneous wave equation with
zero initial data but inhomogeneous Dirichlet boundary conditions and
a suitable test function $v$ we can write Green's first formula as
\[
  \int_0^T \int_\Omega \partial_{n_x} u \, v \, dx \, dt
  = \int_0^T \int_\Omega \Big[ \partial_{tt} u \, v + \nabla_x u \cdot
  \nabla_x v \Big] \, dx \, dt \, .
\]
In particular, for $v=\partial_t u$, this results in the energy representation
\begin{align*}
  E(u) :&= \int_0^T \int_\Omega \Big[ \partial_{tt} u \, \partial_t u +
  \nabla_x u \cdot \nabla_x \partial_t u \Big] \, dx \, dt \\
  &= \frac{1}{2} \, \| \partial_t u(T) \|^2_{L^2(\Omega)} +
  \frac{1}{2} \, \| \nabla_x u(T) \|^2_{L^2(\Omega)} > 0 \, .\nonumber 
\end{align*}
Note that this representation is the basis of the energetic BEM,
see, e.g., \cite{ADG08}.
Instead, when using the particular test function
$ v = {\mathcal{H}}_T u $ and Proposition \ref{proposition Hilbert}
this gives
  \[
    \int_0^T \int_\Gamma \frac{\partial}{\partial n_x} u \,
    {\mathcal{H}}_T u \, ds_x dt =
    \int_0^T \int_\Omega \Big[ {\mathcal{H}}_T \partial_t u \, \partial_t u
    + \nabla_x u \cdot {\mathcal{H}}_T \nabla_x u  \Big] \, dx \, dt \geq 0 \, .
  \]
  Specifically, for the single layer potential $ u = {\mathscr{S}} w$ in
  ${\mathbb{R}}^{n+1} \backslash \Sigma$ we then conclude
  \[
    \langle w , {\mathcal{H}}_T \OV w \rangle_\Sigma =
    \int_0^T \int_\Omega \Big[ {\mathcal{H}}_T \partial_t u \, \partial_t u
    + \nabla_x u \cdot {\mathcal{H}}_T \nabla_x u  \Big] \, dx \, dt \geq 0 \, .
  \]
  In fact, when considering the spatially one-dimensional case $n=1$ we can
  prove the following ellipticity estimate
  \cite{SUZ21,SUT20}
  \[
    \langle w , {\mathcal{H}}_T \OV w \rangle_\Sigma \geq c_1^V \,
    \| w \|^2_{[H^{1/2}_{,0}(\Sigma)]'} \quad
    \mbox{for all} \; w \in [H^{1/2}_{,0}(\Sigma)]' .
  \]
\end{remark}

\noindent
Since the single layer boundary integral operator
$\OV : [H^{1/2}_{,0}(\Sigma)]' \to {\mathcal{H}}_{0,}(\Sigma)$ is invertible,
we can write the solution of the boundary integral equation
\eqref{BIE V} as
\[
  w = \gamma_N^i u = \OV^{-1} (\frac{1}{2}\OI+\OK) g = \OS_i g,
\]
representing the Dirichlet to Neumann map with the interior
Steklov--Poincar\'e operator
\[
  \OS_i = \OV^{-1} (\frac{1}{2}\OI+\OK) :
  {\mathcal{H}}_{0,}(\Sigma) \to [H^{1/2}_{,0}(\Sigma)]' \, .
\]
Hence we find that
\[
  \OV \OS_i = \frac{1}{2} \OI + \OK : {\mathcal{H}}_{0,}(\Sigma) \to
  {\mathcal{H}}_{0,}(\Sigma)
\]
is invertible. As we can formulate a related boundary integral
equation also for the exterior Dirichlet boundary value problem,
\[
\OV \gamma_N^e u = (- \frac{1}{2}\OI+\OK) g \quad \mbox{on} \; \Sigma,
\]
this gives that the exterior Steklov--Poincar\'e operator
\[
  \OS_e = - \OV^{-1} (\frac{1}{2}\OI-\OK) :
  {\mathcal{H}}_{0,}(\Sigma) \to [H^{1/2}_{,0}(\Sigma)]'
\]
is invertible, and so is
\[
  \frac{1}{2} \OI - \OK = - \OV \OS_e :
  {\mathcal{H}}_{0,}(\Sigma) \to
  {\mathcal{H}}_{0,}(\Sigma) \, .
\]
Consequently
\[
  \OV \OW = (\frac{1}{2}\OI - \OK) (\frac{1}{2}\OI+\OK) :
  {\mathcal{H}}_{0,}(\Sigma) \to {\mathcal{H}}_{0,}(\Sigma),
\]
and thus
\[
  \OW = \OV^{-1} (\frac{1}{2}\OI - \OK) (\frac{1}{2}\OI+\OK) :
  {\mathcal{H}}_{0,}(\Sigma) \to [H^{1/2}_{,0}(\Sigma)]' \, .
\]
This finally implies that the hypersingular boundary integral operator
$\OW : {\mathcal{H}}_{0,}(\Sigma) \to [H^{1/2}_{,0}(\Sigma)]'$
satisfies the inf-sup stability condition
\begin{equation}\label{inf-sup W}
    c_1^{\OW} \, \| v \|_{{\mathcal{H}}_{0,}(\Sigma)} \leq
    \sup\limits_{0 \neq \eta \in H^{1/2}_{,0}(\Sigma)}
    \frac{|\langle \OW v , \eta \rangle_\Sigma|}
    {\| \eta \|_{H^{1/2}_{,0}(\Sigma)}} \quad \mbox{for all} \;
    v \in {\mathcal{H}}_{0,}(\Sigma).
\end{equation}
The solution of the Neumann boundary value problem
\[
  \Box u = 0 \quad \mbox{in} \; Q, \quad
  \partial_{n_x} u = \lambda \quad \mbox{on} \; \Sigma, \quad
  u=\partial_tu=0 \quad \mbox{on} \; \Sigma_0
\]
is given by the representation formula
\[
  u(x,t) = (\mathscr{S}\lambda)(x,t) - (\mathscr{D}z)(x,t) \quad
  \mbox{for} \; (x,t) \in Q,
\]
where we can determine the yet unknown Dirichlet datum
$z = \gamma_\Sigma^i u \in {\mathcal{H}}_{0,}(\Sigma)$ as the unique
solution of the first kind boundary integral equation
\[
\OW z = (\frac{1}{2}\OI-\OK) \lambda \quad \mbox{on} \; \Sigma .
\]
Unique solvability follows as described above.

\section{Conclusions}
In this paper, we presented a new framework to describe the
mapping properties of boundary integral operators for the wave equation.
The results are similar as known for the boundary integral operators
for elliptic partial differential equations, i.e., providing ellipticity
and boundedness with respect to function spaces of the same Sobolev spaces.
This will be the starting point to derive quasi-optimal error estimates
for related boundary element methods which are not available so far, and
which will be reported in forthcoming work. Other topics of interest
include efficient implementations of the proposed scheme using
the modified Hilbert transformation, a posteriori error estimates and
adaptivity, an efficient solution of the resulting linear systems of algebraic
equations, and the coupling with space-time finite element methods.

\bibliographystyle{plain}
\bibliography{references}

\end{document}

%% file: macros.tex
\makeatletter
\let\orgdescriptionlabel\descriptionlabel
\renewcommand*{\descriptionlabel}[1]{%
  \let\orglabel\label
  \let\label\@gobble
  \phantomsection
  \edef\@currentlabel{#1}%
  \let\label\orglabel
  \orgdescriptionlabel{#1}%
}
\makeatother

\newcommand{\abs}[1]{\left\lvert #1\right\rvert}

\newcommand{\OV}{\operatorname{\mathsf{V}}}

\newcommand{\OC}{\operatorname{\mathsf{C}}}
\newcommand{\OK}{\operatorname{\mathsf{K}}}
\newcommand{\OW}{\operatorname{\mathsf{W}}}

\newcommand{\OS}{\operatorname{\mathsf{S}}}

\newcommand{\calH}{\mathcal{H}}

\newcommand{\Vx}{\mathbf{x}}

\newcommand{\R}{\mathbb{R}}

\newcommand{\N}{\mathbb{N}}

\newcommand{\dual}[2]{\left\langle #1\,,\,#2\right\rangle}

\definecolor{orange}{rgb}{1,0.4,0}
\definecolor{green}{rgb}{0,0.4,0}

%% file: paper_plain.bbl
\begin{thebibliography}{10}

\bibitem{ADG11}
A.~Aimi, M.~Diligenti, and C.~Guardasoni.
\newblock On the energetic {G}alerkin boundary element method applied to
  interior wave propagation problems.
\newblock {\em J. Comput. Appl. Math.}, 235(7):1746--1754, 2011.

\bibitem{ADG09}
A.~Aimi, M.~Diligenti, C.~Guardasoni, I.~Mazzieri, and S.~Panizzi.
\newblock An energy approach to space-time {G}alerkin {BEM} for wave
  propagation problems.
\newblock {\em Internat. J. Numer. Methods Engrg.}, 80(9):1196--1240, 2009.

\bibitem{ADG08}
A.~Aimi, M.~Diligenti, C.~Guardasoni, and S.~Panizzi.
\newblock A space-time energetic formulation for wave propagation analysis by
  {BEM}s.
\newblock {\em Riv. Mat. Univ. Parma (7)}, 8:171--207, 2008.

\bibitem{BHD86}
A.~Bamberger and T.~Ha Duong.
\newblock Formulation variationnelle pour le calcul de la diffraction d'une
  onde acoustique par une surface rigide.
\newblock {\em Math. Methods Appl. Sci.}, 8(4):598--608, 1986.

\bibitem{BGN16}
L.~Banz, H.~Gimperlein, Z.~Nezhi, and E.~P. Stephan.
\newblock Time domain {BEM} for sound radiation of tires.
\newblock {\em Comput. Mech.}, 58(1):45--57, 2016.

\bibitem{Cia13}
P.~G. Ciarlet.
\newblock {\em Linear and Nonlinear Functional Analysis with Applications}.
\newblock SIAM, Philadelphia, PA, 2013.

\bibitem{COS88}
M.~Costabel.
\newblock Boundary integral operators on {L}ipschitz domains: elementary
  results.
\newblock {\em SIAM J. Math. Anal.}, 19(3):613--626, 1988.

\bibitem{COS90}
M.~Costabel.
\newblock Boundary integral operators for the heat equation.
\newblock {\em Integral Equations Operator Theory}, 13(4):498--552, 1990.

\bibitem{COS94}
M.~Costabel.
\newblock Developments in boundary element methods for time-dependent problems.
\newblock In {\em Problems and methods in mathematical physics ({C}hemnitz,
  1993)}, volume 134 of {\em Teubner-Texte Math.}, pages 17--32. Teubner,
  Stuttgart, 1994.

\bibitem{CoS17}
M.~Costabel and F.-J. Sayas.
\newblock Time-dependent problems with the boundary integral equation method.
\newblock In E.~Stein, R.~Borst, and T.~J.~R. Hughes, editors, {\em
  Encyclopedia of Computational Mechanics Second Edition}. Wiley, 2017.

\bibitem{GaN16}
M.~J. Gander and M.~Neum\"{u}ller.
\newblock Analysis of a new space-time parallel multigrid algorithm for
  parabolic problems.
\newblock {\em SIAM J. Sci. Comput.}, 38(4):A2173--A2208, 2016.

\bibitem{GNS17}
H.~Gimperlein, Z.~Nezhi, and E.~P. Stephan.
\newblock A priori error estimates for a time-dependent boundary element method
  for the acoustic wave equation in a half-space.
\newblock {\em Math. Methods Appl. Sci.}, 40(2):448--462, 2017.

\bibitem{GOS19}
H.~Gimperlein, C.~\"{O}zdemir, D.~Stark, and E.~P. Stephan.
\newblock hp-version time domain boundary elements for the wave equation on
  quasi-uniform meshes.
\newblock {\em Comput. Methods Appl. Mech. Engrg.}, 356:145--174, 2019.

\bibitem{GOS20}
H.~Gimperlein, C.~\"{O}zdemir, D.~Stark, and E.~P. Stephan.
\newblock A residual a posteriori error estimate for the time-domain boundary
  element method.
\newblock {\em Numer. Math.}, 146(2):239--280, 2020.

\bibitem{GOS18}
H.~Gimperlein, C.~\"{O}zdemir, and E.~P. Stephan.
\newblock Time domain boundary element methods for the {N}eumann problem: error
  estimates and acoustic problems.
\newblock {\em J. Comput. Math.}, 36(1):70--89, 2018.

\bibitem{HAD03}
T.~Ha-Duong.
\newblock On retarded potential boundary integral equations and their
  discretisation.
\newblock In {\em Topics in computational wave propagation}, volume~31 of {\em
  Lect. Notes Comput. Sci. Eng.}, pages 301--336. Springer, Berlin, 2003.

\bibitem{HQS17}
M.~E. Hassell, T.~Qiu, T.~S\'{a}nchez-Vizuet, and F.-J. Sayas.
\newblock A new and improved analysis of the time domain boundary integral
  operators for the acoustic wave equation.
\newblock {\em J. Integral Equations Appl.}, 29(1):107--136, 2017.

\bibitem{JoR17}
P.~Joly and J.~Rodr\'{\i}guez.
\newblock Mathematical aspects of variational boundary integral equations for
  time dependent wave propagation.
\newblock {\em J. Integral Equations Appl.}, 29(1):137--187, 2017.

\bibitem{Lad85}
O.~A. Ladyzhenskaya.
\newblock {\em The boundary value problems of mathematical physics}, volume~49
  of {\em Applied Mathematical Sciences}.
\newblock Springer, New York, 1985.

\bibitem{LIM72i}
J.-L. Lions and E.~Magenes.
\newblock {\em Non-homogeneous boundary value problems and applications. {V}ol.
  {I}}.
\newblock Die Grundlehren der mathematischen Wissenschaften, Band 181.
  Springer, New York-Heidelberg, 1972.

\bibitem{LIM72ii}
J.-L. Lions and E.~Magenes.
\newblock {\em Non-homogeneous boundary value problems and applications. {V}ol.
  {II}}.
\newblock Die Grundlehren der mathematischen Wissenschaften, Band 182.
  Springer, New York-Heidelberg, 1972.

\bibitem{McL00}
W.~McLean.
\newblock {\em Strongly elliptic systems and boundary integral equations}.
\newblock Cambridge University Press, Cambridge, UK, 2000.

\bibitem{PoS21}
D.~P\"olz and M.~Schanz.
\newblock On the space-time discretization of variational retarded potential
  boundary integral equations.
\newblock 2021.
\newblock arXiv, 2103.16841v1.

\bibitem{Say13}
F.-J. Sayas.
\newblock Energy estimates for {G}alerkin semidiscretizations of time domain
  boundary integral equations.
\newblock {\em Numer. Math.}, 124(1):121--149, 2013.

\bibitem{Say16}
F.-J. Sayas.
\newblock {\em Retarded potentials and time domain boundary integral equations.
  A road map}, volume~50 of {\em Springer Series in Computational Mathematics}.
\newblock Springer, Cham, 2016.

\bibitem{ScS09}
C.~Schwab and R.~Stevenson.
\newblock Space-time adaptive wavelet methods for parabolic evolution problems.
\newblock {\em Math. Comput.}, 78(267):1293--1318, 2009.

\bibitem{Ste15}
O.~Steinbach.
\newblock Space-time finite element methods for parabolic problems.
\newblock {\em Comput. Meth. Appl. Math.}, 15(4):551--566, 2015.

\bibitem{SUZ21}
O.~Steinbach, C.~Urz\'ua-Torres, and M.~Zank.
\newblock Towards coercive boundary element methods for the wave equation.
\newblock 2021.
\newblock In preparation.

\bibitem{StY18}
O.~Steinbach and H.~Yang.
\newblock Comparison of algebraic multigrid methods for an adaptive space-time
  finite-element discretization of the heat equation in 3{D} and 4{D}.
\newblock {\em Numer. Linear Algebra Appl.}, 25(3):e2143, 17, 2018.

\bibitem{StZ20}
O.~Steinbach and M.~Zank.
\newblock Coercive space-time finite element methods for initial boundary value
  problems.
\newblock {\em Electron. Trans. Numer. Anal.}, 52:154--194, 2020.

\bibitem{StZ21}
O.~Steinbach and M.~Zank.
\newblock A generalized inf-sup stable variational formulation for the wave
  equation.
\newblock 2021.
\newblock arXiv, 2101.06293v1.

\bibitem{SUT20}
Carolina {Urz\'ua-Torres (joint work with O.~Steinbach)}.
\newblock A new approach to time domain boundary integral equations for the
  wave equation.
\newblock {\em Oberwolfach Reports}, 17(1):371--373, 2021.

\bibitem{Yosida:1980}
K.~Yosida.
\newblock {\em Functional analysis}.
\newblock Classics in Mathematics. Springer, Berlin, 1995.

\bibitem{Zan19}
M.~Zank.
\newblock {\em Inf-Sup Stable Space-Time Methods for Time-Dependent Partial
  Differential Equations}, volume~36 of {\em Monographic Series TU
  Graz/Computation in Engineering and Science}.
\newblock Verlag der TU Graz, Graz, 2019.

\end{thebibliography}
